\documentclass[12pt]{article}

\usepackage{epsf}

\usepackage{amsthm}
\usepackage{amssymb}
\usepackage{amsmath}
\usepackage{extarrows}
\usepackage{xcolor}

\usepackage[top=1.5cm,bottom=1.5cm,left=1.5cm,right=1.5cm,includehead,includefoot
           ]{geometry}

\makeatletter

\numberwithin{equation}{section}
 \newtheorem{theorem}{Theorem}[section]
\newtheorem{lemma}[theorem]{Lemma}
\newtheorem{corollary}[theorem]{Corollary}
\newtheorem{remark}[theorem]{Remark}
\newtheorem{definition}[theorem]{Definition}

\newtheorem{proposition}[theorem]{Proposition}

\title{\bf Localizations and Essential Commutant of Toeplitz Algebra on Polydisk}
\author{\large  Jingming Zhu, Chaohua Zhang$^{\diamond}$
\footnote{
College of Data Science, Jiaxing University, Jiaxing, 314001, P.R.China.
E-mail: jingmingzhu@zjxu.edu.cn,\
$^{\diamond}$ College of Data Science, Jiaxing University, Jiaxing, 314001, P.R.China.
E-mail: nyxzch@163.com}}
\date{}

\begin{document}
\maketitle \abstract
%
%
%
%

Usually, the norm closure of a family of operators is not equal to the $C^*$-algebra generated by this family of operators. But, similar with the Bergman space $L^2_a(\textbf{B}, dv)$ of the unit ball in $\mathbb{C}^n$, we show that the norm closure of $\{T_f : f\in L^{\infty}(\mathbb{D}, dv)\}$ on Bergman space $L^2_a(\mathbb{D}, dv)$ of the ploydisk $\mathbb{D}$ in $\mathbb{C}^n$ actually coincides with the Toeplitz algebra $\mathcal{T}(\mathbb{D})$. A key ingredient in the proof is the class of operators $\mathcal{D}$ recently introduced by Yi Wang and Jingbo Xia. In fact, as a by-product, we simultaneously proved that $\mathcal{T}(\mathbb{D})$ also coincides with $\mathcal{D}$. Based on these results, we further proved that the essential commutant of Toeplitz algebra $\mathcal{T}(\mathbb{D})$ equals to $\{T_g: g\in VO_{bdd}\} + \mathcal{K}$ where $VO_{bdd}$ is the collection of functions of vanishing oscillation on polydisk $\mathbb{D}$ and $\mathcal{K}$ denotes the collection of compact operators on $L^2_a(\mathbb{D}, dv)$. On the other hand, we also prove that the essential commutant of $\{T_g: g\in VO_{bdd}\}$ is $\mathcal{T}(\mathbb{D})$, which implies that image of $\mathcal{T}(\mathbb{D})$ in the Calkin algebra satisfies the double commutant relation: $\pi(\mathcal{T}(\mathbb{D}))=\pi(\mathcal{T}(\mathbb{D}))''$.\\


\section{Introduction}
\

Suppose that $\mathcal{Z}$ is a collection of bounded operators on a Hilbert space $\mathcal{H}$. Then its essential commutant is defined to be
$$EssCom(\mathcal{Z}) = \{A \in B(\mathcal{H}):[A, T]\text{ is compact for every }T\in \mathcal{Z}\}.$$

The study of essential commutants began with the classic papers of Johnson-Parrott \cite{Parrott1972}, Voiculescu \cite{Voiculescu1976} and Popa \cite{S. Popa 1987}. Ever since, essential-commutant problems have become a mainstay of operator theory and operator algebras. Recently, base on the localization of the Toeplitz operator which would be described below, Jingbo Xia gave the description of the essential commutant of the Toeplitz algebra \cite{J. Xia 2017} and the double essential commutant of Toeplitz algebra \cite{J. Xia 2018} in the case of Bergman space on unit ball in $\mathbb{C}^n$. Later, Jingbo Xia and Yi Wang generalized the results in the case of strongly pseudo-convex domains with smooth boundary in $\mathbb{C}^n$ \cite{Y. Wang 2020}. But we should point out that unlike the s-weakly localized operators used in \cite{J. Xia 2017} and \cite{J. Xia 2018}, Jingbo Xia and Yi Wang used the collection of operators with the form of sum of tensor product of normalized reproducing kernel with bounded coefficients, which would be denoted by $\mathcal{D}$ below, as technical tool and proved that $\mathcal{D}$ is contained in Toeplitz algebra $\mathcal{T}$. Inspired by the results above, we will consider the case of polydisk $\mathbb{D}$ in $\mathbb{C}^n$.

Let
\[
\mathbb{D}=\{z=(z_1,\cdots,z_n)\in \mathbb{C}^n:|z_i|\leq 1\text{ for any }i=1,\cdots,n\}.
\]
For $n=1$, we denote $\mathbb{D}=D$. Let $dv$ denote the normalized Lebesgue Volume measure on $\mathbb{D}$, i.e.
$$dv(z)=dA_1(z_1)\cdots dA_n(z_n),$$
 in which $dA_i$ is the normalized area measure on $D$. Then the normalized measure defined by the formula
$$d\Lambda(z)=d\Lambda_1(z_1)\cdots d\Lambda_n(z_n)=\frac{dA_1(z_1)\cdots dA_n(z_n)}{(1-|z_1|^2)^2\cdots (1-|z_n|^2)^2}$$
is M\"obius invariant on $\mathbb{D}$. Recall that the Bergman space $L_a^2({\mathbb{D}})$ is the closure subspace
$$\{h\in L^2(\mathbb{D},dv): h\text{ is analytic on }\, \mathbb{D}\}.$$
of $L^2(\mathbb{D},dv)$. It is well known that $L_a^2({\mathbb{D}})$ is a reproducing kernel Hilbert space with reproducing kernel given by the formula
$$K_z(w)=\prod \limits_{i=1}^n\frac{1}{(1-w_i\overline{z_i})^2},\text{ for any } z=(z_1,\dots,z_n),\,w=(w_1,\dots,w_n)\in \mathbb{D}.$$
 Note that
 \begin{equation}\label{norm}
 \|K_z\|=\prod_{i=1}^n (1-|z_i|^2).
 \end{equation}
 As usual, let $k_z$ denote the normalized reproducing kernel for $L_a^2({\mathbb{D}})$, that is,
 $$k_z(w)=\prod \limits_{i=1}^n\frac{1-|z_i|^2}{(1-w_i\overline{z_i})^2},\text{ for any }  z=(z_1,\dots,z_n),\,w=(w_1,\dots,w_n)\in \mathbb{D}.$$

Let $P:L^2(\mathbb{D},dv)\rightarrow L_a^2({\mathbb{D}})$ be the Bergman projection. For each $f\in L^{\infty}(\mathbb{D},dv)$, the following formula gives rise to the Toeplitz operator
$$T_fh=P(fh),\text{ for any } h\in L_a^2({\mathbb{D}}).$$
Let $\mathcal{T}(\mathbb{D})$ be the $C^*$-algebra generated by $\{T_f:f\in L^{\infty}(\mathbb{D},dv)\},$ the collection of Toeplitz operators with bounded symbols. For any $z,w\in \mathbb{D}$, let $\varphi_z(w)$ be the M\"{o}bius transform in polydisk defined by
$$
\varphi_z(w)=(\varphi_{z_1}(w_1),\cdots, \varphi_{z_n}(w_n)),
$$
where $\varphi_{z_i}(w_i)=(z_i-w_i)(1-w_i\bar{z}_i)^{-1}\text{ for any }1\leq i\leq n$ is the M\"obius  transform of the unit disc $D$.

Recall that the Bergman metric on $D$ is given by the formula
$$\beta(u,v)=\frac12\log\frac{1+|\varphi_u(v)|}{1-|\varphi_u(v)|},\quad u,v\in D.$$
For ploydisk $\mathbb{D}$, we will use the following metric defined by
$$d(z,w)=\sum_{i=1}^n \beta(z_i,w_i),\text{ for any }  z=(z_1,\dots,z_n),\,w=(w_1,\dots,w_n)\in \mathbb{D},$$
where $\beta(z_i,w_i)$ is the Bergman metric on $D$ for each $i\in\{1,\cdots,n\}$.

For any $z=(z_1,\dots,z_n),\,w=(w_1,\dots,w_n)\in \mathbb{D}$,
\begin{equation}\label{indentity}
|\langle k_z,k_w \rangle|=\prod_{i=1}^n\frac{(1-|z_i|^2)(1-|w_i|^2)}{|1-w_i\bar{z}_i|^2}=\prod_{i=1}^n (1-|\varphi_{z_i}(w_i)|^2).
\end{equation}
A direct computation shows that for any $u,v \in D$,
$$\beta(u,v)=\frac12\log\frac{1}{1-|\varphi_u(v)|^2}+c\text{ for some }0<c<\log 2,$$
that is, $1-|\varphi_u(v)|^2=e^{2c}e^{-2\beta(u,v)}$.
Combining this with (\ref{indentity}), we have that for any $z,\,w\in \mathbb{D}$, there exists a $0<C<\infty$ such that
\begin{equation}\label{inequality}
|\langle k_z,k_w \rangle| = Ce^{-2d(z,w)}.
\end{equation}

\begin{itemize}
  \item For any $u\in D$ and $a>0$, we denote the $\beta$-ball by $B(z,a)=\{w\in D:\beta(u,w)<a\};$
  \item For any $z\in \mathbb{D}$ and $a>0$, we denote the $d$-ball by $D(z,a)=\{w\in\mathbb{D}:d(z,w)<a\};$
  \item For any $z\in \mathbb{D}$ and $a>0$, we denote the $d$-square by $$\widetilde{D}(z,a)=\{w\in\mathbb{D}:\beta(z_i,w_i)<a\text{ for any }i=1,\cdots,n\}.$$
\end{itemize}
By the the M\"{o}bius invariance of $\beta$ and $d\Lambda$, for any $a>0$, we have
\begin{equation}\label{invariance}
\Lambda(\widetilde{D}(z,a))=\prod_{i=1}^n \Lambda_i(B(z_i,a))=\prod_{i=1}^n \Lambda_i(\varphi_{z_i}(B(0,a)))=\prod_{i=1}^n \Lambda_i((B(0,a))=\Lambda(\widetilde{D}(0,a)).
\end{equation}
To state our results, we need the notion of `localized' operators of the following form on polydisk $\mathbb{D}$ in section 3, which was first introduced in \cite{Y. Wang 2020} for strongly pseudo-convex domains.

\begin{definition}\label{Definition 0}
(a) Let $\mathcal{D}_{0,0}$ denote the collection of operators of the form
$$\sum_{z\in\Gamma}c_zk_z\otimes k_z,$$
where $\Gamma$ is any separated set in $\mathbb{D}$, $\{c_z:z\in \Gamma\}$ is any bounded set of complex coefficients.

(b) Let $\mathcal{D}_0$ denote the collection of operators of the form
$$\sum_{z\in\Gamma}c_zk_z\otimes k_{\gamma(z)},$$
where $\Gamma$ is any separated set in $\mathbb{D}$, $\{c_z:z\in \Gamma\}$ is any bounded set of complex coefficients, and $\gamma:\Gamma\to \mathbb{D}$ is any map for which there is a $0\leq C<\infty$ such that for every $z\in\Gamma$,
$$d(z,\gamma(z))\leq C.$$

(c) Let $\mathcal{D}$ denote the closure of the linear span of $\mathcal{D}_0$ with respect to the operator norm.

(d) Let $C^*(\mathcal{D})$ denote the $C^*$-algebra generated by $\mathcal{D}$.
\end{definition}

As it turns out, similar techniques in \cite{Y. Wang 2020} would give our first main result for classic compactness criterion for $A\in C^*(\mathcal{D})$, in terms of its Berezin transform.\\

\noindent
{\bf Theorem A:}
For $A\in C^*(\mathcal{D})$,
\begin{equation}\label{G1}
  \lim_{z\to\partial \mathbb{D}}\langle Ak_z,k_z\rangle=0 \text{ if and only if }A\in\mathcal{K}.
\end{equation}

Although the compactness criterion in Theorem A is seemed to be different with the classic compactness criterion for $A\in \mathcal{T}$ in \cite{Y. Wang 2020} and \cite{Suarez 2007}, our second main result will fill the gap.\\

\noindent
{\bf Theorem B:}
For Toeplitz algebra $\mathcal{T}(\mathbb{D})$ on polydisk, we have $C^*(\mathcal{D})=\mathcal{T}(\mathbb{D})$.\\

To state our last two main results, we need the notion of vanishing oscillation, which was first introduced in \cite{Berger1988}, \cite{CoburnZhu1990} for functions on bounded symmetric domains with respect to the Bergman metric.

\begin{definition}\label{Definition 2}
 For any $f\in L^{\infty}(\mathbb{D},dv)$ and any $t\geq 0$, we denote
$$\mathrm{diff}_t(f)=\sup\{|f(z)-f(w)|:\text{ for any } z,w\in\mathbb{D}\text{ with }d(z,w)\le 1 \text{ and }d(0,z)\geq t\}.$$
We denote $\mathrm{diff}_0(f)$ by $\mathrm{diff}(f)$ for simplicity.
\end{definition}

\begin{definition}\label{Definition 3}
A continuous function $f$ on $\mathbb{D}$ is said to have vanishing oscillation if
$$\lim_{t\to\infty}\mathrm{diff}_t(f) = 0.$$
We denote by $VO_{bdd}(\mathbb{D})$ the collection of continuous functions having vanishing oscillation on $\mathbb{D}$ that are also bounded.
\end{definition}

Our last two main results are the two theorems below:\\

\noindent
{\bf Theorem C:}
For polydisk $\mathbb{D}\subset \mathbb{C}^n$, $EssCom\{T_g:g\in VO_{bdd}\}= \mathcal{T}(\mathbb{D}).$\\

\noindent
{\bf Theorem D:}
For polydisk $\mathbb{D}\subset \mathbb{C}^n$, $EssCom(\mathcal{T}(\mathbb{D}))=\{T_g:g\in VO_{bdd}\}+\mathcal{K}$.\\

Let $\mathcal{Q}$ denote the Calkin algebra $B(L^2_a(\mathbb{D}))/\mathcal{K}$. For any $\mathcal{X}\subset \mathcal{Q}$, write $\mathcal{X}'= \{Y\in \mathcal{Q}: [Y,X] = 0 \text{ for every }X \in \mathcal{X}\}$ for its commutant in $\mathcal{Q}$. Let $\pi: B(L^2_a(\mathbb{D}))\to \mathcal{Q}$ be the quotient homomorphism. Then $\pi(EssCom(\mathcal{Z})) = \{\pi(\mathcal{Z})\}'$ for every subset $\mathcal{Z}\subset B(L^2_a(\mathbb{D}))$. Obviously, a subset $\mathcal{A}$ of $\mathcal{Q}$ satisfies the double-commutant relation $\mathcal{A} = \mathcal{A}''$ if and only if $\mathcal{A} = \mathcal{G}'$ for some $\mathcal{G}\subset\mathcal{Q}$. Thus Theorem C and D implies that $\pi(\mathcal{T}(\mathbb{D}))$ satisfies the double-commutant relation in $\mathcal{Q}$.\

We conclude the Introduction by summarizing the main steps in the proofs.\

First of all, in Section 2 we presented some techniques to make the parameter for separation flexible. Section 2 also contains some basic norm estimation for the operators of tensor product of normalized reproducing kernels. Section 3 brings in another important ingredient for our analysis, the collection of operators $\mathcal{D}$. The main result of this section is that $C^*(\mathcal{D})=\mathcal{D}$, which would be used in the proof of Theorem A and B.\

Sections 4 is devoted to the radial-spherical decomposition, which supplies a method to decompose operators in the forms of sum of tensor product of normalized reproducing kernels in the Proof of Theorem A in next section. In section 5, we investigate the membership of some collections of operators and proved that $\mathcal{D}_0\subset \mathcal{T}^{(1)}$ and $\{T_{f}:f\in L^{\infty}(\mathbb{D})\}\subset C^*(\mathcal{D})$, which would be used in the proof of Theorem B.

In section 6, we introduce $LOC(A)$, the class of localized versions of $A$ for any bounded operator $A$ on $L_a^2(\mathbb{D})$, and proved that $LOC(A)\subset \mathcal{D}(A)$, in which $\mathcal{D}(A)$ is the class of localized versions of $A$ which is `localized' by Toeplitz operators. Section 7 is devoted to matters called $\varepsilon-\delta$ condition. In section 7, we proved that any operators in $EssCom\{T_g:g\in VO_{bdd}\}$ would satisfy  $\varepsilon-\delta$ condition. While in section 8, by the estimation from radial-spherical decomposition, we have that any bounded operator $X$ in $EssCom\{T_g:g\in VO_{bdd}\}$ belongs to the norm closure of $Span\{LOC(X)+\mathcal{K}\}$.\

With all the preparation above, we finally prove Theorem C and D in section 9. The gist of the proof is that, combined with the fact $LOC(A)\subset \mathcal{D}(A)$ and $\mathcal{D}_0\subset \mathcal{T}^{(1)}$ proved in previous sections, the $\varepsilon-\delta$ condition of the operators in $EssCom\{T_g:g\in VO_{bdd}\}$ would ensure its membership of $\mathcal{T}(\mathbb{D})$ in Remark 9.10. While some basic estimation $\|[M_f,P]\|\leq C \text{diff}(f)$ would easily drive to the other inclusion, which completes the proof of Theorem C. Then with the work in Section 9, this is now relatively easy to prove Theorem D. First of all, $EssCom(\mathcal{T}(\mathbb{D}))\supset\{T_g:g\in VO_{bdd}\}$ is already proven in the proof of Theorem C. Then for the opposite direction, the membership $A\in  EssCom(\mathcal{T}(\mathbb{D}))$ implies that the Berezin transform
$\widetilde{A}$ of $A$ belongs to $VO_{bdd}$. Since $A-T_{\widetilde{A}}\in \mathcal{T}(\mathbb{D})$ , the membership $\widetilde{A}\in VO_{bdd}$ and the work in Section 9 lead to an easy proof of the membership $A-T_{\widetilde{A}} \in \mathcal{K}$, which proves Theorem D.

\textbf{Acknowledgment:} We would like to thank Jingbo Xia for many inspiring discussions. We appreciate all members of Functional analysis Seminar in Jiaxing University very much for their helpful comments and suggestions.

\section{Separated Sets}
\

The technical details begin with

\begin{definition}\label{separatedsets}
For $a>0$, a subset $\Gamma$ of $\mathbb{D}$ is said to be $a-$separated if $D(u,a)\cap D(v,a)= \emptyset$ holds for all $u \neq v$ in $\Gamma$. And a subset $\Gamma$ of $\mathbb{D}$ is said to be separated, if $\Gamma$ is $a-$separated for some $a>0$.
\end{definition}

The following lemma is an analog of Lemma 2.2 in \cite{J. Xia 2015}, which shows that there is a bound for cardinality of bounded sets in $\Gamma$ and we can adjust the parameters for separation of the image of M\"{o}bius transform by taking finite partitions. Since the proof can be translated word by word from the proof of Lemma 2.2 in \cite{J. Xia 2015}, we omit the proof here.

\begin{lemma}\label{Lemma 0.1}Let $\Gamma$ be a separated set in $\mathbb{D}$.

(a) For each $0<R<\infty$, there is an $N_0=N_0(\Gamma, R)\in \mathbb{N}$ (the set of zero and all positive integers) such that $$\mathrm{card}\{w\in \Gamma:\beta(z_i,w_i)\le R,\,\text{ for any }i=1,\cdots,n\}\le N_0,\text{ for every } \,z\in \Gamma.$$

(b) For every $0<a<\infty$, there is a partition $\Gamma=\Gamma_1\sqcup\dots \sqcup \Gamma_m$ such that for every $1\le i\le m$, $\Gamma_i$ is $a$-separated.

(c) For every pair of $z\in \mathbb{D}$ and $\rho>0$, there exists a finite partition $\Gamma=\Gamma_1\sqcup\dots \sqcup \Gamma_m$ such that for every $j\in \{1,\dots,m\}$, any two distinct elements $w,\xi \in \Gamma_j$,
$$\beta(\varphi_{w_i}(z_i),\varphi_{{\xi}_i}(z_i))>\rho,\,\text{ for any }i=1,\cdots,n.$$
\end{lemma}

The following norm estimation would be used frequently in this paper.

\begin{proposition}\label{Proposition A}
Let $\Gamma$ be a separated set in $\mathbb{D}$ and let $\{c_z:z\in \Gamma\}$ be a subset in $\mathbb{C}$ with $\mathop{\mathrm{sup}}\limits_{z\in \Gamma}|c_z|< \infty$. Then  there is a constant $0<C<\infty$ such that for any orthonormal set $\{e_z:z\in \Gamma\}$,
\[
\|\sum_{z\in\Gamma}c_zk_z\otimes e_z\|\leq C\sup_{z\in \Gamma}|c_z|.
\]
\end{proposition}

\begin{proof}
By Lemma \ref{Lemma 0.1} (b), we can assume without loss of generality that $\Gamma$ is 1-separated set in $\mathbb{D}$.
Since $\Gamma$ is 1-separated, then, for any $z\neq w$ in $\Gamma$, we have
\begin{equation}\label{bujiaoA}
\widetilde{D}(z,\frac1n)\cap \widetilde{D}(w,\frac1n)=\emptyset.
\end{equation}
Write $F=\sum_{z\in\Gamma}c_zk_z\otimes e_z$. Then $F^*F=\sum_{z,w\in\Gamma}\overline{c}_w c_z\langle k_z,k_w\rangle e_w\otimes e_z$
. We will estimate the operator norm of $F^*F$ by using Schur's test with $h(z)=\prod \limits_{i=1}^n(1-|z_i|^2)^{\frac12}$. For any $z\in\Gamma$, we have

\begin{eqnarray*}
  \sum_{w\in\Gamma}|\overline{c}_wc_z|\cdot|\langle k_z,k_w\rangle| h(w)&=& \sum_{w\in\Gamma}|\overline{c}_wc_z|\cdot \prod_{i=1}^n\frac{(1-|z_i|^2)(1-|w_i|^2)^{\frac{3}{2}}}{|1-w_i\bar{z}_i|^2} \\
   &\le & (\mathop{\mathrm{sup}}\limits_{z\in \Gamma}|c_z|)^2\sum_{w\in\Gamma}\prod_{i=1}^n\frac{1}{\Lambda_i(B(w_i,\frac1n))}\int_{B(w_i,\frac1n))}\frac{(1-|z_i|^2)(1-|\xi_i|^2)^{\frac{3}{2}}}{|1-\xi_i\bar{z}_i|^2}d\Lambda_i(\xi_i) \\
 &= & (\mathop{\mathrm{sup}}\limits_{z\in \Gamma}|c_z|)^2 \sum _{w\in \Gamma}\frac{1}{\Lambda(\widetilde{D}(w,\frac1n))}\int_{\widetilde{D}(w,\frac1n)} |\langle k_z,k_{\xi}\rangle|\cdot \|K_{\xi}\|^{\frac12} d\Lambda(\xi)\\
 &\le & (\mathop{\mathrm{sup}}\limits_{z\in \Gamma}|c_z|)^2 \frac{1}{\Lambda(\widetilde{D}(0,\frac1n))}\int_{\mathbb{D}} |\langle k_z,k_{\xi}\rangle|\cdot \|K_{\xi}\|^{\frac12} d\Lambda(\xi)\\
   &\leq & (\mathop{\mathrm{sup}}\limits_{z\in \Gamma}|c_z|)^2 \frac{1}{\Lambda(\widetilde{D}(0,\frac1n))} \prod_{i=1}^n\int_{D}\frac{(1-|z_i|^2)}{|1-\xi_i\bar{z}_i|^2(1-|\xi_i|^2)^{\frac{1}{2}}}dA_i(\xi_i),
\end{eqnarray*}
in which the first inequality follows from the Lemma 4.11 in \cite{Zhu2007}, the second inequality follows from the M\"{o}bius invariance of d$\Lambda$ and (\ref{bujiaoA}) and the second equality follows as $
\|K_{\xi}\|=\prod_{i=1}^n (1-|{\xi}_i|^2).
$

Combined with Rudin-Forelli estimates, it gives that there exists $0<C_1<\infty$ such that
\[
\sum_{w\in\Gamma}|\overline{c}_wc_z|\cdot|\langle k_z,k_w\rangle|h(w)\leq (C_1\mathop{\mathrm{sup}}\limits_{z\in \Gamma}|c_z|)^2 \prod_{i=1}^n\frac{(1-|z_i|^2)}{(1-|z_i|^2)^{\frac{1}{2}}}=(C_1\mathop{\mathrm{sup}}\limits_{z\in \Gamma}|c_z|)^2h(z).
\]
Similarly, we have that there exists $0<C_2<\infty$ such that
\[
\sum_{w\in\Gamma}|\overline{c}_wc_z|\cdot |\langle k_z,k_w\rangle|h(z)\leq (C_2\mathop{\mathrm{sup}}\limits_{z\in \Gamma}|c_z|)^2 h(w).
\]
It follows from Schur's test that $$\|F\|^2=\|F^*F\|\le C_1C_2(\mathop{\mathrm{sup}}\limits_{z\in \Gamma}|c_z|)^2.$$
\end{proof}
Recall that each Toeplitz operator on Polydisk $\mathbb{D}$ has an `integral representation' in terms of normalized reproducing kernel $\{k_w:w\in\mathbb{D}\}$. For the orthogonal projection $P$ on Bergman space $L_a^2(\mathbb{D})$, we have
\begin{equation}\label{-1}
P=\int_{\mathbb{D}}k_z\otimes k_zd\Lambda(z),
\end{equation}
which is the identity operator on $L_a^2(\mathbb{D})$.

Let $\Delta$ be a subset of $\mathbb{D}$ which is maximal with respect to the property of being $1-$separated. Then we have
\[
\mathbb{D}\subset \bigcup_{z\in\Delta}D(z,2)\subset \bigcup_{z\in\Delta}\widetilde{D}(z,2)
\]
by the maximality.

For the fixed $\Delta$, applying Lemma \ref{Lemma 0.1} (a) and the property of $1-$separated set $\Delta$, there is a natural number $N_{\Delta}\in\mathbb{N}$ such that
\begin{equation}\label{B2}
\mathrm{card}\{z\in\Delta:z\in \widetilde{D}(w,4)\}\leq N_{\Delta} \text{ for any }w\in\Delta,
\end{equation}
which ensures $\mathrm{card}\{z\in\Delta:w\in \widetilde{D}(z,2)\}\leq N_{\Delta} \text{ for any }w\in\mathbb{D}$.

Now for fixed $\Delta$, we define
\begin{equation}\label{B3}
A_{\Delta}=\sum_{z\in\Delta}T_{\chi_{\widetilde{D}(z,2)}}=\sum_{z\in\Delta}\int_{\widetilde{D}(z,2)}k_w\otimes k_w d\Lambda(w) = \sum_{z\in\Delta}\int_{\widetilde{D}(0,2)}k_{\varphi_z(w)}\otimes k_{\varphi_z(w)} d\Lambda(w),
\end{equation}
where the third equality holds by the M\"{o}bius invariance of $\beta$ and $d\Lambda$.
By (\ref{B2}) and (\ref{B3}), we have $1\leq A_{\Delta}\leq N_{\Delta}$, then the self-adjoint Toeplitz operator $A_{\Delta}$ is invertible with $\|A_{\Delta}^{-1}\|\le 1$.

For each $z\in\mathbb{D}$, define
\[
(U_zf)(w)=f(\varphi_z(w))k_z(w)\text{ for every }f\in L_a^2(\mathbb{D}).
\]
Then $U_z$ is a unitary and $U_zU_z=1$.
 A direct computation shows that
\begin{equation}\label{yousuanzi} U_zk_w=g_{z,w}k_{\varphi_z(w)},\text{ for any } z=(z_1,\dots,z_n),\,w=(w_1,\dots,w_n)\in \mathbb{D},\end{equation}
where
$$g_{z,w}=\prod_{i=1}^n\frac{|1-z_i\overline{w_i}|^2}{(1-z_i\overline{w_i})^2}.$$
Then we can rewrite \begin{equation}\label{yousuanzi2}k_{\varphi_z(w)}\otimes k_{\varphi_z(w)}=(U_zk_w)\otimes (U_zk_w).\end{equation} And the operator $A_{\Delta}$ can be represented as
\begin{equation}\label{B45}
A_{\Delta}=\int_{\widetilde{D}(0,2)}G_wd\Lambda(w),
\end{equation}
in which
\begin{equation}\label{B4}
G_w=\sum_{z\in \Delta}k_{\varphi_z(w)}\otimes k_{\varphi_z(w)}=\sum_{z\in \Delta}(U_zk_w)\otimes (U_zk_w).
\end{equation}

\begin{lemma}\label{Lemma 1} Let $G$ be a separated set in $\mathbb{D}$. Then there is a $0<C<\infty$ such that for any subset $\{h_z:z\in G\}$ of the collection of bounded analytic functions on $\mathbb{D}$ and any orthonormal set $\{u_z:z\in G\}$,
\[
\|\sum_{z\in G}(U_zh_z)\otimes u_z\|\leq C\mathop{\mathrm{sup}} \limits_{z\in G}\|h_z\|_{\infty}.
\]
\end{lemma}
\begin{proof}
By Lemma \ref{Lemma 0.1} (b), we can assume that $\Gamma$ is 1-separated set in $\mathbb{D}$ and $\mathop{\mathrm{sup}} \limits_{z\in G}\|h_z\|_{\infty}<\infty$ in the statement.
Write $T=\sum_{z\in G}(U_zh_z)\otimes u_z$. Since the operator $A_{\Delta}$ in (\ref{B3}) is invertible with $\|A_{\Delta}^{-1}\|\le 1$, we have that $\|T\|\le \|A_{\Delta}^{-1}A_{\Delta}T\|\le \|A_{\Delta}T\|$. Combining this with (\ref{B45}), we have $$\|T\|\le \Lambda(\widetilde{D}(0,2))\mathop{\mathrm{sup}} \limits_{w\in \widetilde{D}(0,2)}\|G_wT\|,$$
where $G_w$ is the operator that appear in (\ref{B4}).
So it suffices to estimate the norm of $G_wT$ for $w\in \widetilde{D}(0,2).$

Let $\{e_z:z\in \Delta\}$ be the orthonormal set. Then the operator $G_w$ can be written as $G_w=H_wH_w^*$, where
\begin{equation}\label{B5}
  H_w=\sum_{z\in\Delta}(U_zk_w)\otimes e_z=\sum_{z\in\Delta}(g_{z,w}k_{\varphi_z(w)})\otimes e_z.
\end{equation}
By Proposition \ref{Proposition A} and Lemma \ref{Lemma 0.1} (c), there is a $C_1>0$ such that $\|H_w\|\leq C_1$. It suffices to prove that there exists $C>0$ such that
\[
\|H_w^*T\|\leq C\mathop{\mathrm{sup}} \limits_{z\in G}\|h_z\|_{\infty},\text{ for any }w\in \widetilde{D}(0,2).
\]

We will estimate the operator norm of $H_w^*T=\sum_{a\in\Delta}\sum_{z\in G}\langle U_zh_z,U_ak_w\rangle e_a\otimes u_z$
by Schur's test for any $w\in\widetilde{D}(0,2)$. But firstly, we need to estimate the coefficients $\{\langle U_zh_z,U_ak_w\rangle\}_{a\in\Delta,z\in G}$.

Since $U_zh_z=h_z\circ\varphi_z\cdot k_z$ and $k_z\circ \varphi_a=g_{a,z} \frac{k_{\varphi_a(z)}}{k_a}$, then for any $z\in G$, $a\in \Delta$ and $w\in \widetilde{D}(0,2)$,
\begin{multline}\label{UzhzUakw}
\langle U_zh_z,U_ak_w\rangle =\langle U_a(U_zh_z),k_w\rangle =\langle U_a(h_z\circ \varphi_z\cdot k_z),k_w\rangle \\
=\langle (h_z\circ \varphi_z\circ \varphi_a \cdot k_z\circ \varphi_a)\cdot k_a ,k_w\rangle =g_{a,z}\cdot d_{z,a}(w)\langle   k_{\varphi_a(z)},k_w\rangle
\end{multline}
where $g_{a,z}=\prod_{i=1}^n\frac{|1-a_i\overline{z_i}|^2}{(1-a_i\overline{z_i})^2}$, $d_{z,a}(w)=h_z\circ \varphi_z\circ \varphi_a(w)$.

Recall that
$
\prod_{i=1}^n (1-|\varphi_{a_i}(z_i)|^2)=|\langle k_a, k_z \rangle|.
$
A direct computation shows that
\begin{equation}\label{kazw}
\langle k_{\varphi_a(z)},k_w\rangle= \prod_{i=1}^n\frac{(1-|\varphi_{a_i}(z_i)|^2)(1-|w_i|^2)}{(1-\overline{\varphi_{a_i}(z_i)}w_i)^2}=|\langle k_a, k_z \rangle|\cdot k_{\bar{w}}(\overline{\varphi_a(z)}).
\end{equation}
By Rudin-Forelli estimates and (\ref{norm}), we have that there exists a $0<C_2<\infty$ such that
\begin{eqnarray*}
\|K_{w}\|^{\frac{2}{3}}|k_{\bar{w}}(\overline{\varphi_a(z)})|&=& \prod_{i=1}^n\frac{(1-|w_i|^2)^{\frac53}}{|1-\overline{\varphi_{a_i}(z_i)}w_i|^2} \\
&\leq& \prod_{i=1}^n \frac{1}{\Lambda_i(B(0,2))}\int_{B(0,2)}\frac{|1-|\xi_i|^2|^{\frac{5}{3}}}{|1-\overline{\varphi_{a_i}(z_i)}\xi_i|^2}\frac{dA_i(\xi_i)}{(1-|\xi_i|^2)^2}\\
&\leq& C_2\prod_{i=1}^n|1-|\varphi_{a_i}(z_i)|^2|^{-\frac{1}{3}}
=C_2|\langle k_a,k_z\rangle |^{-\frac{1}{3}},
\end{eqnarray*}
where the first inequality follows from the Lemma 4.11 in \cite{Zhu2007} and integration in polar coordinates.
For any $w\in \widetilde{D}(0,2)$, it is easy to check that $\|K_{w}\|^{-1}\le e^{4n}$ by the definition of Bergman metric. Write $C_3=C_2\cdot e^{\frac83 n}.$
Then we have $|k_{\bar{w}}(\overline{\varphi_a(z)})|\leq C_3|\langle k_a,k_z\rangle |^{-\frac{1}{3}}$.
Combining this with (\ref{UzhzUakw}) and (\ref{kazw}), we have
\begin{eqnarray*}
|\langle U_zh_z,U_ak_w\rangle|&\leq&  \mathop{\mathrm{sup}} \limits_{z\in G}\|h_z\|_{\infty} \cdot |\langle k_{\varphi_a(z)},k_w\rangle|\\
&\le& \mathop{\mathrm{sup}} \limits_{z\in G}\|h_z\|_{\infty} \cdot|\langle k_a, k_z \rangle |\cdot |k_{\bar{w}}(\overline{\varphi_a(z)})| \\
 &\leq& C_3 \mathop{\mathrm{sup}} \limits_{z\in G}\|h_z\|_{\infty}\cdot |\langle k_a,k_z\rangle |^{\frac{2}{3}}.
\end{eqnarray*}
Therefore, to estimate the norm of $H_w^*T$, we only need to estimate the norm of the opertor $\sum_{a\in\Delta}\sum_{z\in G}\mathop{\mathrm{sup}} \limits_{z\in G}\|h_z\|_{\infty}(\langle k_a,k_z\rangle )^{\frac{2}{3}}e_a\otimes u_z$.
By an estimate similar to $\|F^*F\|$ in the proof of Proposition \ref{Proposition A} with the testing function $h(z)=\prod \limits_{i=1}^n(1-|z_i|^2)^{\frac12}$, there exists $C_4>0$, such that for any $w\in \widetilde{D}(0,2)$,
 $\|F_w\|\le C_4\cdot \mathop{\mathrm{sup}} \limits_{z\in G}\|h_z\|_{\infty}.$
Then we have that for any $w\in \widetilde{D}(0,2)$, $\|H_w^*T\| \leq C_3C_4\mathop{\mathrm{sup}} \limits_{z\in G}\|h_z\|_{\infty}$, which completes the proof.
\end{proof}

\section{The Operators in $C^*(\mathcal{D})$}
\

For the operators defined in Definition \ref{Definition 0}, we have

\begin{lemma}\label{Corollary 1}
If the operator $T$ belongs $\mathcal{D}_{0}$, then $T$ is bounded on the Bergman space $L_a^2(\mathbb{D})$.
\end{lemma}
\begin{proof}
For any operator $T\in \mathcal{D}_0$, we rewrite
$$T=\sum_{u\in \Gamma}c_uk_u\otimes k_{\gamma(u)},$$
where $\Gamma$ is a $\delta$-separated set in $\mathbb{D}$, $\mathrm{sup}_{u\in \Gamma}|c_u|<\infty$ and $\gamma:\Gamma\rightarrow \mathbb{D}$ is any map for which there exists a $0\le C < \infty$ such that $d(z,\gamma(z))\le C$ for every $z\in \Gamma$.

We can decompose
$$T=\sum_{u\in \Gamma}c_uk_u\otimes k_{\gamma(u)}=AB,$$
in which $A=\sum_{u\in \Gamma}c_uk_u\otimes e_u$, $B=\sum_{u\in \Gamma} e_u\otimes k_{\gamma(u)}$ and $\{e_u:u\in \Gamma\}$ is an orthonormal set.

By Lemma \ref{Lemma 0.1} (b), we can decompose
$$\Gamma=\Gamma_1\sqcup\dots \sqcup \Gamma_m,$$
such that $\Gamma_i$ is $3C$-separated for every $i\in\{1,\cdots,m\}$. Then the set $\{\gamma(z):z\in \Gamma_i\}$ is $C$-separated for every $i\in\{1,\cdots,m\}$. Therefore,
$$B=\sum_{i=1}^m B_i,$$
in which $B_i=\sum_{u\in \Gamma_i} e_u\otimes k_{\gamma(u)}$ for every $i\in\{1,\cdots,m\}$. Combining $\|B\|\le \sum_{i=1}^m \|B_i\|$ with Proposition \ref{Proposition A}, we have that $A$ and  $B$  both are bounded operators and hence $T$ is also a bounded operator.
\end{proof}

In fact, if the coefficients $|c_u|$ tends to 0 as $u$ goes to $\partial \mathbb{D}$, then it is obvious that

\begin{lemma}\label{corollary 2} For any operator $T\in \mathcal{D}_0$, we can assume that the operator $$T=\sum_{u\in \Gamma}c_uk_u\otimes k_{\gamma(u)},$$
where $\Gamma$ is a $\delta$-separated set in $\mathbb{D}$, $\mathrm{sup}_{u\in \Gamma}|c_u|<\infty$ and $\gamma:\Gamma\rightarrow \mathbb{D}$ is any map for which there exists a $0\le C < \infty$ such that $d(z,\gamma(z))\le C$ for every $z\in \Gamma$.  If $|c_u|$ tends to 0 as $u$ goes to $\partial \mathbb{D}$, then $T$ is a compact operator on the Bergman space $L_a^2(\mathbb{D})$.
\end{lemma}

\begin{proposition}\label{Proposition B}
Span $(\mathcal{D}_{0,0})$ contains an invertible operator on $L_a^2(\mathbb{D})$.
\end{proposition}
\begin{proof}
Consider the operator $A_{\Delta}=\int_{\widetilde{D}(0,2)}G_wd\Lambda(w)$ which appears in (\ref{B45}), where
$$
G_w=\sum_{z\in \Delta}k_{\varphi_z(w)}\otimes k_{\varphi_z(w)}=\sum_{z\in \Delta}(U_zk_w)\otimes (U_zk_w).
$$
Since we can always assume that $\{\varphi_z(w):  z\in\Delta\}$ is separated by Lemma \ref{Lemma 0.1} (c), then we have $G_w\in $ Span $(\mathcal{D}_{0,0})$. By Lemma \ref{Lemma 1}, as in (\ref{B5}), we can rewrite $G_w=H_wH_w^*$ where $H_w=\sum_{z\in\Delta}(U_zk_w)\otimes e_z$ satisfies that for any $w,w'\in \widetilde{D}(0,2)$, there exists $C_B>0$ such that
\[
\|H_w-H_{w'}\|=\|\sum_{z\in\Gamma}(U_z(k_w-k_{w'}))\otimes e_z\|\leq C_B\|k_w-k_{w'}\|_{\infty}.
\]
For each $z \in \mathbb{D}$, it is obvious that
\[
\lim_{w\to z}\|k_z - k_w\|_{\infty}= 0.
\]
Combined with $G_w=H_wH_w^*$, it gives that the map $w\mapsto G_w$ is continuous with respect to norm $\|\cdot\|$. Then there is a Riemann sum $S$ of the operator $A_{\Delta}=\int_{\widetilde{D}(0,2)}G_wd\Lambda(w)$, such that $\|1-A_{\Delta}^{-1}S\|< \frac{1}{2}$, so $S$ is also invertible, which completes the proof.
\end{proof}
Obviously, we have the following corollaries.

\begin{corollary}\label{Corollary 2}
If $A$ is an operator such that $XAY$ is a compact operator for every $X,Y\in \mathcal{D}_{0,0}$, then $A$ is a compact operator.
\end{corollary}

\begin{corollary}\label{Corollary 3}
If $A$ is an operator such that $XAY\in C^*(\mathcal{D})$ for all $X,Y\in \mathcal{D}_{0,0}$, then $A\in C^*(\mathcal{D})$.
\end{corollary}

\begin{lemma}\label{Lemma C}
 Given any $\varepsilon>0$, there is an $r>0$ such that the following holds:

For a subset $E$ of $\Gamma\times G$, with $\Gamma, G$ being $1-$separated in $\mathbb{D}$, satisfying the condition $d(z,w)\geq r$ for every $(z,w)\in E$, if the coefficient set $\{a_{z,w}:(z,w)\in E\}$ satisfies $|a_{z,w}|\leq |\langle k_z,k_w\rangle|^{1-t}$ for every $(z,w)\in E$ for some $0\leq t<\frac{1}{2}$, then for any orthonormal sets $\{e_z:z\in \Gamma\}$ and $\{u_w:w\in G\}$, we have
\[
\|\sum_{(z,w)\in E}a_{z,w}e_z\otimes u_w\|\leq \varepsilon.
\]
\end{lemma}
\begin{proof}
Given $\{e_z:z\in \Gamma\}$ and $\{u_z:z\in G\}$ as in the statement, let $E_r$ be a subset of $\Gamma\times G$ satisfying the condition $d(z,w)\geq r$ for every $(z,w)\in E$. Write $T_r=\sum_{(z,w)\in E_r}a_{z,w}e_z\otimes u_w$.
We will estimate the operator norm of $T_r$ by Schur's test with $h(z)=\prod \limits_{i=1}^n(1-|z_i|^2)^{\frac12}$.

Let $t<\eta<\frac{1}{2}$. By the formula (\ref{inequality}), we have that for any given $z\in \Gamma$, and the given $\varepsilon>0$, there is an $r_0>0$ such that
\[
|\langle k_z,k_w\rangle|^{\eta-t}< \Lambda(\widetilde{D}(0,\frac1n))\varepsilon\text{ for any }w\in G\text{ with }d(z,w)\geq r_0.
\]
Moreover, $|a_{z,w}|\leq |\langle k_z,k_w\rangle|^{1-t}\le \Lambda(\widetilde{D}(0,\frac1n))\, \varepsilon \, |\langle k_z,k_w\rangle|^{1-\eta}.$
Then for any given $z\in \Gamma$,
\begin{eqnarray*}
& &  \sum_{w\in G,d(z,w)\geq r_0}|a_{z,w}|h(w) \\
&\leq & \sum_{w\in G,d(z,w)\geq r_0}\Lambda(\widetilde{D}(0,\frac1n))\varepsilon |\langle k_z,k_w\rangle|^{1-\eta}h(w) \\
                                            &= & \Lambda(\widetilde{D}(0,\frac1n))\varepsilon \sum_{w\in G,d(z,w)\geq r_0}\frac{1}{\Lambda(\widetilde{D}(w,\frac1n))}\int_{\widetilde{D}(w,\frac1n)}|\langle k_z,k_{\xi}\rangle|^{1-\eta}h(\xi)d\Lambda(\xi)\\
             &\le& \varepsilon  \int_{\mathbb{D}}|\langle k_z,k_{\xi}\rangle|^{1-\eta}h(\xi)d\Lambda(\xi)\\
                                            &\leq & \varepsilon \prod_{i=1}^n\int_{D}\frac{(1-|z_i|^2)^{1-\eta}}{|1-\xi_i\bar{z}_i|^{2-2\eta}(1-|\xi_i|^2)^{\frac{1}{2}+\eta}}dA_i(\xi_i)\\
                                            &\leq & \varepsilon C_1\prod_{i=1}^n(1-|z_i|^2)^{1-\eta}(1-|z_i|^2)^{-\frac{1}{2}-\eta+2\eta}= \varepsilon C_1 h(z), \text{for some}\,0<C_1<\infty,
\end{eqnarray*}
where the second
inequality follows by the M\"{o}bius invariance of $d\Lambda$ and the fact that for any $z\neq w$ in $G$, $\widetilde{D}(z,\frac1n)\cap \widetilde{D}(w,\frac1n)=\emptyset$ by the condition that $G$ is $1-$separated set in $\mathbb{D}$, and the fourth inequality follows from the Rudin-Forelli estimate in the unit disk $D$.
Similarly, for any $w\in G$, we have $\sum_{z\in \Gamma,d(z,w)\geq r_0}|a_{z,w}|h(z)\leq \varepsilon C_2 h(w)$, for some $0<C_2<\infty$. By Schur's test, we have that $\|T_r\|\le \sqrt{C_1C_2}\varepsilon$ for any $r \ge r_0$, which completes the proof.
\end{proof}

\begin{proposition}\label{Proposition D}
$C^*(\mathcal{D})=\mathcal{D}$.
\end{proposition}
\begin{proof}
It suffices to prove that for any $A,B\in \mathcal{D}_0$, we have $A^*$ and $AB\in \mathcal{D}$.

We can assume
$$A=\sum_{z\in\Gamma}c_zk_z\otimes k_{\gamma(z)}\text{ and }B=\sum_{w\in G}d_wk_w\otimes k_{\phi(w)},$$
 in which $\Gamma,G$ are $1-$separated sets in $\mathbb{D}$, $\sup_{z\in\Gamma}|c_z|<\infty$, $\sup_{w\in G}|d_w|<\infty$ and there exists $C>0$ such that $d(z,\gamma(z))<C$ and $d(w,\phi(w))<C$ for any $z\in\Gamma$ and $w\in G$.

Firstly, we will prove that $A^*\in$ Span$(\mathcal{D}_0)$. By Lemma \ref{Lemma 0.1} (b), we can decompose $\Gamma$ as the disjoint union as $\Gamma=\Gamma_1\sqcup\cdots\sqcup\Gamma_l$ with $\Gamma_j$ being $(3C)-$separated for each $j\in\{1,\cdots,l\}$. Then for each $j\in\{1,\cdots,l\}$, the set $\{\gamma(z):z\in \Gamma_j\}$ is $C$-separated. That is, $\sum_{z\in \Gamma_j}\bar{c}_z k_{\gamma(z)}\otimes k_z\in \mathcal{D}_0$ for each $j\in\{1,\cdots,l\}$. And we can also decompose
\[
A^*=\sum_{j=1}^l\sum_{z\in \Gamma_j}\bar{c}_z k_{\gamma(z)}\otimes k_z,
\]
which implies $A^*\in $ Span$(\mathcal{D}_0)$.

Next for each $R>0$, we can decompose $AB=T_{R,1}+T_{R,2}$, in which
\[
T_{R,1}=\sum_{(z,w)\in \Gamma\times G,d(z,w)\leq R}c_zd_w\langle k_w,k_{\gamma(z)}\rangle k_z\otimes k_{\phi(w)};
~~
T_{R,2}=\sum_{(z,w)\in \Gamma\times G,d(z,w)> R}c_zd_w\langle k_w,k_{\gamma(z)}\rangle k_z\otimes k_{\phi(w)}.
\]
By the decomposition above, it suffices to prove that:
\begin{itemize}
  \item $T_{R,1}\in $ Span$(\mathcal{D}_0)$,
  \item For any $\varepsilon>0$, there is a $R>0$ such that $\|T_{R,2}\|\leq\varepsilon$.
\end{itemize}

(1) For $T_{R,1}$, denote $E=\{(z,w)\in\Gamma\times G:d(z,w)\leq R\}$. Since $G$ is $1-$separated, by Lemma \ref{Lemma 0.1} (a), there is $N_0\in \mathbb{N}$ such that for any $z\in\Gamma$,
\begin{equation}\label{D1}
\text{card}\{w\in G:d(z,w)\leq R\}\leq N_0.
\end{equation}
Let $E_1$ be a maximal subset of $E$ with respect to the property:
\[
(z,w)\neq(z',w')\in E_1 \text{ implies }z\neq z'.
\]
Inductively, suppose that we have found $E_1,\cdots, E_v$. Then let $E_{v+1}$ be the maximal subset in $E\setminus (E_1\sqcup\cdots\sqcup E_v)$  with respect to the property above. After $k$ $(k\le N_0)$ steps, we have disjoint subsets $E_1,\dots, E_{k}$. Suppose there are some $(z,w)\in E\setminus (E_1\sqcup\cdots\sqcup E_{N_0})$. By the maximality of each $E_j$, there is a $(z,w^{(j)})\in E_j$ such that $w\neq w^{(j)}$, which contradicts the (\ref{D1}). Therefore we have a finite partition  $E=E_1\sqcup\cdots \sqcup E_{k}$ for some $k\le N_0$ such that each $E_j$ satisfies the property above for $j\in\{1,\cdots,k\}$. Then for each $j\in\{1,\cdots,k\}$,
\[
\sum_{(z,w)\in E_j}c_zd_w\langle k_w,k_{\gamma(z)}\rangle k_z\otimes k_{\phi(w)}\in \mathcal{D}_0,
\]
and hence
$$T_{R,1}=\sum_{(z,w)\in E}c_zd_w\langle k_w,k_{\gamma(z)}\rangle k_z\otimes k_{\phi(w)}=\sum_{j=1}^{k}\sum_{(z,w)\in E_j}c_zd_w\langle k_w,k_{\gamma(z)}\rangle k_z\otimes k_{\phi(w)}\in \text{ Span}(\mathcal{D}_0 ).$$

(2) We can write $T_{R,2}=X\widetilde{T}_{R,2}Y$, where
\[
X=\sum_{z\in\Gamma}c_zk_z\otimes e_z,~~Y=\sum_{w\in G}d_wu_w\otimes k_{\phi(w)},~~\widetilde{T}_{R,2}=\sum_{(z,w)\in \Gamma\times G,d(z,w)> R}\langle k_z,k_{\phi(w)}\rangle e_z\otimes u_w
\]
and $\{e_z:z\in\Gamma\}$, $\{u_w:w\in G\}$ are orthonormal sets.

By Proposition \ref{Proposition A}, there exists $C_1>0$ such that
\[
\|X\|\leq C_1\sup_{z\in \Gamma}|c_z|\text{ and }\|Y\|\leq C_1\sup_{w\in G}|d_w|.
\]
By (\ref{inequality}), there exists $C_2>0$ such that $|\langle k_z,k_{\phi(w)}\rangle|=C_2e^{-2d(z,\phi(w))}$. By the assumption $d(w,\phi(w))<C$ for any $w\in G$, we have $d(z,\phi(w))\geq d(z,w)-C$. Therefore,
\[
|\langle k_z,k_{\phi(w)}\rangle|=C_2e^{-2d(z,\phi(w))}\leq C_2e^{2C-2d(z,w)}\leq C_2'|\langle k_z,k_w\rangle|\,\text {for some}\, C_2'>0.
\]
By Lemma \ref{Lemma C}, it gives that there exists a $R>0$ such that
\[
\|\widetilde{T}_{R,2}\|\leq \frac{\varepsilon}{C_1^2\sup_{z\in\Gamma}|c_z|\sup_{w\in G}|d_w|},
\]
and hence $\|T_{R,2}\|\leq \|X\|\cdot \|\widetilde{T_{R,2}}\| \cdot\|Y\|\leq \varepsilon$.

\end{proof}

\begin{lemma}\label{Lemma E}
 Given any $A\in C^*(\mathcal{D})$ and $\varepsilon>0$, there is an $r=r(A,\varepsilon)>0$ such that the following holds:\
Let $\Gamma,G$ be $1-$separated sets in $\mathbb{D}$ and let $\{e_z:z\in\Gamma\}$ and $\{u_w:w\in G\}$ be orthonormal sets. Define
\[
X=\sum_{z\in\Gamma}e_z\otimes k_z\text{ and }Y=\sum_{w\in G}k_w\otimes u_w.
\]
If $\Gamma,G$ satisfy $d(z,w)\geq r$ for any $(z,w)\in\Gamma\times G$, then $\|XAY\|\leq \varepsilon$.
\end{lemma}

\begin{proof}
By Propositions \ref{Proposition A}, we can assume $\|X\|\leq C$ and $\|Y\|\leq C$ for some $C>0$. According to Proposition \ref{Proposition D}, for any $A\in C^*(\mathcal{D})$ and $\varepsilon>0$, there exists $\widetilde{A}\in$ Span$(\mathcal{D}_0)$ such that
\[
\|A-\widetilde{A}\|<\frac{\varepsilon}{C^2}
\]
and hence
\[
\|X(A-\widetilde{A})Y\|\leq \|X\|\|A-\widetilde{A}\|\|Y\|< C\frac{\varepsilon}{C^2}C=\varepsilon.
\]
Therefore we only need to prove that for any $A\in\mathcal{D}_0$, $\|XAY\|< \varepsilon$.

For $A\in\mathcal{D}_0$, we can assume
\[
A=\sum_{u\in M}c_uk_u\otimes k_{\gamma(u)}
\]
in which $M$ is $\delta-$separated for some $\delta>0$, $\sup_{u\in M}|c_u|<\infty$ and there exists $C_1>0$ such that $d(u,\gamma(u))\leq C_1$ for any $u\in M$. Then
\begin{eqnarray*}
  XAY &=& \sum_{(z,w)\in \Gamma\times G}\sum_{u\in M}c_u\langle k_u,k_z\rangle\langle k_w,k_{\gamma(u)}\rangle e_z\otimes u_w\\
      &=& \sum_{(z,w)\in \Gamma\times G} a_{z,w} e_z\otimes u_w, \text{ where }a_{z,w}=\sum_{u\in M}c_u\langle k_u,k_z\rangle\langle k_w,k_{\gamma(u)}\rangle
\end{eqnarray*}
By (\ref{inequality}), there exists $C_2>0$ such that $|\langle k_u,k_z\rangle\langle k_w,k_{\gamma(u)}\rangle|\leq C_2e^{-2d(u,z)}e^{-2d(w,\gamma(u))}$.

Combined with the inequalities
$d(w,\gamma(u))\geq d(u,w)-C_1$, this gives that
\[
|\langle k_u,k_z\rangle\langle k_w,k_{\gamma(u)}\rangle|\leq C_2e^{2C_1}e^{-2d(u,z)}e^{-2d(u,w)}.
\]
Again, by (\ref{inequality}), there exists $C_3>0$ such that
\[
|c_u\langle k_u,k_z\rangle\langle k_w,k_{\gamma(u)}\rangle|\leq C_3|\langle k_u,k_z\rangle\langle k_w,k_u\rangle|.
\]
Then
\begin{eqnarray*}
|a_{z,w}|&=&|\sum_{u\in M}c_u\langle k_u,k_z\rangle\langle k_w,k_{\gamma(u)}\rangle| \\
&\leq& \sup_{u\in M}\{|c_u|\}\cdot C_3\sum_{u\in M}|\langle k_u,k_z\rangle\langle k_w,k_u\rangle| \\
&\leq& \sup_{u\in M}\{|c_u|\}\cdot C_3\sum_{u\in M}\prod_{i=1}^n\frac{1}{\Lambda_i(B(u_i,\frac{\delta}{n}))}\int_{B(u_i,\frac{\delta}{n})}\frac{(1-|z_i|^2)(1-|w_i|^2)dA_i(\xi_i)}{|1-z_i\bar{\xi}_i|^2|1-w_i\bar{\xi}_i|^2} \\
 &\leq& \sup_{u\in M}\{|c_u|\}\cdot C_3\cdot \frac{1}{\Lambda(\widetilde{D}(0,\frac{\delta}{n})}\prod_{i=1}^n\int_{D}\frac{(1-|z_i|^2)(1-|w_i|^2)dA_i(\xi_i)}{|1-z_i\bar{\xi}_i|^2|1-w_i\bar{\xi}_i|^2} \\
&=& \sup_{u\in M}\{|c_u|\}\cdot C_3 \cdot \frac{1}{\Lambda(\widetilde{D}(0,\frac{\delta}{n})}\langle|k_z|,|k_w|\rangle \\
&=& \sup_{u\in M}\{|c_u|\}\cdot C_3 \cdot \frac{1}{\Lambda(\widetilde{D}(0,\frac{\delta}{n})}\prod_{i=1}^n\int_{D}\frac{|k_{z_i}(\xi_i)|}{|k_{w_i}(\xi_i)|}|k_{w_i}(\xi_i)|^2dA_i(\xi_i) \\
&=& \sup_{u\in M}\{|c_u|\}\cdot C_3 \cdot \frac{1}{\Lambda(\widetilde{D}(0,\frac{\delta}{n})} \prod_{i=1}^n\int_{D}\frac{|k_{z_i}(\xi_i)|}{|k_{w_i}(\xi_i)|}dA_i(\varphi_{w_i}(\xi_i)) \\
&=& \sup_{u\in M}\{|c_u|\}\cdot C_3\cdot \frac{1}{\Lambda(\widetilde{D}(0,\frac{\delta}{n})} \prod_{i=1}^n\int_{D}\frac{|k_{z_i}(\varphi_{w_i}(\xi_i))|}{|k_{w_i}(\varphi_{w_i}(\xi_i)|}dA_i(\xi_i)\\
&=&\sup_{u\in M}\{|c_u|\}\cdot C_{3}\cdot \frac{1}{\Lambda(\widetilde{D}(0,\frac{\delta}{n})}\prod_{i=1}^n\int_{D}|k_{\varphi_{w_i}(z_i)}(\xi_i)|dA_i(\xi_i)
\end{eqnarray*}
where the third inequality follows from the M\"obius invariance of $d\Lambda$ and the fact that for any $z\neq w$ in $M$, $\widetilde{D}(z,\frac{\delta}{n})\cap \widetilde{D}(w,\frac{\delta}{n})=\emptyset$ 
and the last equality follows from this that for any $u,v,\xi$ in $D$,
$k_u(\varphi_v(\xi))=\frac{k_{\varphi_v(u)}(\xi)}{k_v(\xi)}\left(\frac{|1-\overline{u}v|}{1-\overline{u}v}\right)^2$ by the direct computation. Combined with Rudin-Forelli estimates, it gives that there is a $C_4>0$ and  a $0<\varepsilon <\frac{1}{2}$ such that
\begin{eqnarray*}
|a_{z,w}|&\leq& C_3\prod_{i=1}^n\int_{D}|k_{\varphi_{w_i}(z_i)}(\xi_i)|dA_i(\xi_i)\\
&\leq& C_3\prod_{i=1}^n\int_{D}\frac{1-|\varphi_{w_i}(z_i)|^2}{|1-\overline{\varphi_{w_i}(z_i)}\xi_i|^2(1-|\xi_i|^2)^{\varepsilon}}  dA_i(\xi_i)\\
&\leq& C_4\prod_{i=1}^n(1-|\varphi_{w_i}(z_i)|^2)^{1-\varepsilon}=C_4|\langle k_z,k_w\rangle|^{1-\varepsilon}.
\end{eqnarray*}
By Lemma \ref{Lemma C}, we have proven the result.
\end{proof}

\section{Radial-spherical Decomposition}
\

As a technique to decompose the unit disk, we need the radial-spherical decomposition from \cite{J. Xia 2018}.

Let $S^1$ denote $\{\xi\in\mathbb{C}:|\xi|=1\}$. Recall the formula
\[
\tilde{d}(u,\xi)\doteq|1-\langle u,\xi\rangle |^{\frac{1}{2}},\text{ for any } u,\xi\in S^1
\]
defines a metric on $S^1$. For any $u\in S^1$ and $r>0$, we write $\widetilde{B}(u,r)=\{\xi\in S^1:\tilde{d}(u,\xi)<r\}$.

Let $\sigma$ be the standard spherical measure on $S^1$ with the usual normalization $\sigma(S^1)=1$. By Proposition 5.1.4 in \cite{W. Rudin 1980}, there is a constant $A_0\in (2^{-1},\infty)$, such that
\begin{equation}\label{Delta1}
\min\{2^{-1},\pi^{-1}\}r^2\leq \sigma(\widetilde{B}(u,r))\leq A_0r^2, \text{ for any }u\in S^1, 0<r\leq \sqrt{2}.
\end{equation}
For the radial direction of $D$, we set
\begin{equation}\label{R3}
\rho_m=1-2^{-2m}\text{ for any }m\in\mathbb{N}_+ \text{ (the set of positive integers)}.
\end{equation}
For each pair of natural numbers $m\ge 6$ and $j\in \mathbb{N}_+$, we denote
\[
\alpha_{m,j}=m(1-\rho_{jm}^2)^{\frac{1}{2}}=m\cdot 2^{-jm}(2-2^{-2jm})^{\frac{1}{2}}.
\]
Note that $8\alpha_{m,j}\leq \sqrt{2}$ for all $m\geq 6$ and $j\in\mathbb{N}_+$. For each pair of $m\geq 6$ and $j\in\mathbb{N}_+$, let $E_{m,j}$  be a subset of $S^1$ which is maximal with respect to the property of being $\frac{\alpha_{m,j}}{2}-$separated.

By (\ref{Delta1}), $E_{m,j}$ is finite for all $m\geq 6$ and $j\in\mathbb{N}_+$. It follows from the maximality of $E_{m,j}$ that
\begin{equation}\label{Dleta2}
\bigcup_{u\in E_{m,j}}\widetilde{B}(u,\alpha_{m,j})=S^1.
\end{equation}
For each triple of $m\geq 6$, $j\in\mathbb{N}_+$ and $u\in E_{m,j}$, we define
\begin{equation}\label{R-1}
A_{m,j,u}=\{r\xi:\xi\in B(u,\alpha_{m,j}),r\in[\rho_{(j+2)m},\rho_{(j+3)m}]\};
\end{equation}
\begin{equation}\label{R-2}
B_{m,j,u}=\{r\xi:\xi\in B(u,3\alpha_{m,j}),r\in[\rho_{jm},\rho_{(j+5)m}]\}.
\end{equation}
By (\ref{Dleta2}), we have
\begin{equation}\label{sanjiao1}
\bigcup_{j=1}^{\infty}\bigcup_{u\in E_{m,j}}A_{m,j,u}=\{z\in D:\rho_{3m}\leq |z|<1\}.
\end{equation}
We define
\begin{equation}\label{buchongA}
A_{m,0,0}=\{z\in D:|z|<\rho_{3m}\},~~~B_{m,0,0}=\{z\in D:|z|< \rho_{5m}\}.
\end{equation}
Since $E_{m,j}$ is $\frac{\alpha_{m,j}}{2}-$separated set in $S^1$ for all $m\geq 6$ and $j\in\mathbb{N}_+$, by (\ref{Delta1}), there exists $\widetilde{N}_0=\widetilde{N}_0(\alpha_{m,j})\in\mathbb{N}$ such that for every triple of  $m\geq 6$, $j\in\mathbb{N}_+$ and $u\in E_{m,j}$, we have
\begin{equation}\label{Delta3}
\text{card}\{v\in E_{m,j}:\tilde{d}(u,v)<7\alpha_{m,j}\}\leq \widetilde{N}_0.
\end{equation}

Let $E_{m,j}^{(1)}$ be a subset of $E_{m,j}$ that is maximal with respect to the property that
\[
\tilde{d}(u,v)\geq 7\alpha_{m,j}\text{ for any }u\neq v\in E_{m,j}^{(1)}.
\]
Inductively, suppose that we have defined $E_{m,j}^{(1)},\cdots, E_{m,j}^{(t)}$. Then let $E_{m,j}^{(t+1)}$ be the maximal subset in $E_{m,j}\setminus (E_{m,j}^{(1)}\sqcup\cdots\sqcup E_{m,j}^{(t)})$  with respect to the property above. By (\ref{Delta3}), adding empty sets when necessary, we have $E_{m,j}=E_{m,j}^{(1)}\sqcup\cdots \sqcup E_{m,j}^{(\widetilde{N}_0)}$ such that for any $t\in\{1,\cdots,\widetilde{N}_0\}$, we have
\[
\tilde{d}(u,v)\geq 7\alpha_{m,j}\text{ for any }u\neq v\in E_{m,j}^{(t)}.
\]
The number $\widetilde{N}_0$ and the partition above will be fixed for the rest of the paper. For the convenience of the discussion below, we denote $E_{m,0}=\{0\}$.

The (a), (b) and (c) parts of the following lemma follow from Lemma 2.8 in \cite{J. Xia 2018} and the (d) part is obvious from (c). So the proof is omitted here.

\begin{lemma}\label{Lemma 2}\cite{J. Xia 2018}

(a)  Let $m \geq 6$, $j\in\mathbb{N}_+$ and $t\in\{1,\cdots,\widetilde{N}_0\}$. If $u, v\in E_{m,j}^{(t)}$ and $u \neq v$, then
we have $\beta(z, w) > 2$ for all $z\in B_{m,j,u}$ and $w\in B_{m,j,v}$.

(b) Let $m \geq 6$. If $u \in E_{m,j}$, $v \in E_{m,k}$ and $k\geq j + 6$, then we have $\beta(z, w) > 3$ for all $z\in B_{m,j,u}$ and $w\in B_{m,k,v}$.

(c) Let $m \geq 6$, $j\in\mathbb{N}_+$ and $u\in E_{m,j}$. Then $\beta(z, w) \geq 2 \log m$ for all $z\in D\setminus B_{m,j,u}$ and $w\in A_{m,j,u}$.

(d) Let $m \geq 6$, $j\ge 3$ and $u\in E_{m,j}$. Then $\beta(z,w)\geq 2\log m$ for any $z\in A_{m,0,0}$, $w\in A_{m,j,u}$.\\
\end{lemma}

We will use the following Lemma \ref{Lemma 3}, which is also the Lemma 4.3 in \cite{J. Xia 2018} to give an analog estimate in the following Lemma \ref{Lemma 3'}.

\begin{lemma}\label{Lemma 3}\cite{J. Xia 2018}
For each triple of $m \geq 6$, $j\in\mathbb{N}_+$ and $u \in E_{m,j}$, define
\begin{equation}\label{z}
z_{m,j,u} = \rho_{jm}u.
\end{equation}
Then we have $B_{m,j,u}\subset D(z_{m,j,u}, R_m)$, where $R_m = 2 + 5m + \log (1 + 2^{10m}\cdot 18m^2)$.
\end{lemma}

\begin{lemma}\label{Lemma 3'}
For each triple of $m \geq 6$, $\hat{j}=(j_1,\cdots,j_n)\in\mathbb{N}^n$ and $\hat{u}=(u_1,\cdots,u_n) \in E_{m,\hat{j}}=E_{m,j_1}\times\cdots \times E_{m,j_n}$, define
\[
z_{m,j_i,u_i} = \rho_{j_im}u_i \text{ for each }i\in\{1,\cdots,n\},~~~z_{m,0,0}=0.
\]
Then we have $\hat{B}_{m,\hat{j},\hat{u}}\subset D(z_{m,\hat{j},\hat{u}}, L_m)$, where
$$\hat{B}_{m,\hat{j},\hat{u}}=B_{m,j_1,u_1}\times\cdots\times B_{m,j_n,u_n},~~z_{m,\hat{j},\hat{u}}=(z_{m,j_1,u_1},\cdots,z_{m,j_n,u_n}),$$
 and $L_m = nR_m+n\log(2^{10m+1})$ with $R_m$ appearing in Lemma \ref{Lemma 3}.
\end{lemma}

\begin{proof}
For any $z\in \hat{B}_{m,\hat{j},\hat{u}}$, if $(j_i,u_i)\neq (0,0)$, then by Lemma \ref{Lemma 3}, we have
\[
d(z,z_{m,\hat{j},\hat{u}})=\sum_{i=1}^n\beta(z_i,z_{m,j_i,u_i})\leq nR_m.
\]
Since for any $u\in B_{m,0,0}$,
$$\beta(u,0)=\frac12\log\frac{1+|\varphi_u(0)|}{1-|\varphi_u(0)|}=\frac12\log\frac{1+|u|}{1-|u|}\leq \frac12\log\frac{1+\rho_{5m}}{1-\rho_{5m}}=\frac12\log\frac{2-2^{-10m}}{2^{-10m}}\leq \log(2^{10m+1}),$$
 then $B_{m,0,0}\subset D(z_{m,0,0},\log(2^{10m+1}))$ and hence for $L_m = nR_m+n\log(2^{10m+1})$, we have
\[
\hat{B}_{m,\hat{j},\hat{u}}\subset D(z_{m,\hat{j},\hat{u}},L_m).
\]
\end{proof}

The following key Lemma from \cite{Y. Wang 2020} would be used frequently in this paper.

\begin{lemma}\label{KeyLemma}\cite{Y. Wang 2020}
Let $E_1,\cdots,E_l$ be the orthogonal projections on a Hilbert space $\mathcal{H}$. Suppose that they are mutually orthogonal. Given any $A\in B(\mathcal{H})$, there is a subset $L$ in $\{1,\cdots,l\}$ such that if we define
\[
E=\sum_{j\in L}E_j\text{ and }F=\sum_{k\in \{1,\cdots,l\}\setminus L}E_k,
\]
then
$$\|\sum_{j\neq k}E_jAE_k\|\leq 4(\|EAF\|+\|FAE\|).$$\\
\end{lemma}

\section{Proof of Theorem A and B}
\

As we already mentioned, our goal for this section is to prove Theorem A and B. Firstly, we will introduce the following definition.

\begin{definition}\label{Definition 1} Let $m\geq 6$ be given.

(a) For each pair of $k \in\{0,1, 2, 3, 4, 5\}$ and $\nu \in\{1,\cdots, \widetilde{N}_0\}$, where $\widetilde{N}_0$ is the integer that appears in (\ref{Delta3}), let
\[
I_m^{(\nu,k)}=\{(m, 6j+k, u):j\in\mathbb{N},u\in E_{m,6j+k}^{(\nu)}\} \text{ for }k\neq 0
\]
\[
I_m^{(\nu,k)}=\{(m, 6j+k, u):j\in\mathbb{N},u\in E_{m,6j}^{(\nu)}\} \text{ for }k= 0, \text{ where } E_{m,0}^{(\nu)}=\{0\}.
\]
Write
\[
\hat{\nu}=(\nu_1,\cdots,\nu_n),~~v_i\in\{1,\cdots,\widetilde{N}_0\} \text{ for any }1\leq i\leq n;
\]
\[
\hat{k}=(k_1,\cdots,k_n),~~k_i\in\{0,\cdots,5\} \text{ for any }1\leq i\leq n.
\]
\begin{equation}\label{zhibiao}
\hat{I}_m^{(\hat{\nu},\hat{k})}= (I_m^{(\nu_1,k_1)}\times\cdots\times I_m^{(\nu_n,k_n)})\setminus \{((m,0,0),\cdots,(m,0,0))\}.
\end{equation}
For each $1\leq i\leq n$, $I_m^{(\nu_i,k_i)}=\{(m, 6j_i+k_i, u_i):j_i\in\mathbb{N},u_i\in E_{m,6j_i+k_i}^{(\nu_i)}\}$.

(b)  Let $\hat{I}_m=\bigsqcup_{\hat{\nu}}\bigsqcup_{\hat{k}}\hat{I}_m^{(\hat{\nu},\hat{k})}$.
\end{definition}
In the rest of this paper, for any $\hat{\omega}=((m, 6j_1+k_1, u_1),\dots ,(m, 6j_n+k_n, u_n)) \in \hat{I}_m^{(\hat{\nu},\hat{k})}$, let \begin{equation}\label{liangyong}
\hat{B}_{\hat{\omega}}=B_{m, 6j_1+k_1, u_1}\times \dots \times B_{m, 6j_n+k_n, u_n}\, \text{ and }\, \hat{A}_{\hat{\omega}}=A_{m, 6j_1+k_1, u_1}\times \dots \times A_{m, 6j_n+k_n, u_n}.
\end{equation}
Let $\hat{j}=(j_1,\cdots,j_n)$ and $\hat{u}=(u_1,\cdots,u_n)$. We should note that, in this case, $\hat{B}_{\hat{\omega}}=\hat{B}_{m,\hat{j},\hat{u}}$  and $\hat{A}_{\hat{\omega}}=\hat{A}_{m,\hat{j},\hat{u}}$. By Lemma \ref{Lemma 2} and the definition of $B_{m,0,0}$, it is easy to find that $\hat{B}_{\hat{\omega}}\cap \hat{B}_{\hat{\omega}'}=\emptyset$ when $\hat{\omega}\neq \hat{\omega}'$ in $I_m^{(\hat{\nu},\hat{k})}$ for each $\hat{\nu}\in\{1,\cdots,\widetilde{N}_0\}^n$ and $\hat{k}\in\{0,\cdots,5\}^n$.\\

By the above Radial-spherical decomposition in section 4, it is easy to check that
\begin{equation}\label{quanbing}\mathop\bigcup \limits_{\hat{\omega} \in \hat{I}_m}\hat{A}_{\hat{\omega}}=\{z\in \mathbb{D}:|z_i|<\rho_{3m}\,\text{for each}\, 1\le i\le n\}^{c},\end{equation}
 Write
 \begin{equation}\label{yuji}
 \hat{A}_{m,\hat{0},\hat{0}}\doteq A_{m,0,0}\times \cdots \times A_{m,0,0}=\{z\in \mathbb{D}:|z_i|<\rho_{3m}\,\text{for each}\, 1\le i\le n\}.
 \end{equation}

\begin{lemma}\label{Lemma F}
 For $A\in B(L_a^2(\mathbb{D}))$, if
\begin{equation}\label{*}
\lim_{z\to \partial{\mathbb{D}}}\langle Ak_z,k_z\rangle=0,
\end{equation}
then for any $\varepsilon>0$, we have
\begin{equation}\label{**}
\lim_{z\to \partial{\mathbb{D}}}\sup\{|\langle Ak_z,k_w\rangle|:d(z,w)\leq r\}=0.
\end{equation}
\end{lemma}

\begin{proof}
Suppose (\ref{**}) fails for some $r>0$, then there exists sequences $\{z^{(j)}\},\{w^{(j)}\}$ in $\mathbb{D}$ such that
\begin{itemize}
  \item (1)  $d(z^{(j)},w^{(j)})\leq r$ for any $j\in\mathbb{N}$;
  \item (2)  There exists a $a>0$ such that $|\langle Ak_{z^{(j)}}, k_{w^{(j)}}\rangle|\geq a$ for any $j\in\mathbb{N}$;
  \item (3)   $z^{(j)}\to \partial\mathbb{D}$ as $j\to\infty$.
\end{itemize}

Passing to subsequence if necessary, by Banach-Alaoglu Theorem, we may assume that the weak limit exists and we denote
\[
\widetilde{A}=w-\lim_{j\to\infty}U_{z^{(j)}}^*AU_{z^{(j)}}.
\]
Recall that
\[
(U_{z}h)(\xi)=k_z(\xi)h(\varphi_{z}(\xi)), \text{ for any }\xi\in\mathbb{D}, h\in L_a^2(\mathbb{D}).
\]
For any $w\in \mathbb{D}$, we have
\[
\langle \widetilde{A}k_w,k_w\rangle =\lim_{j\to\infty}\langle U_{z^{(j)}}^*AU_{z^{(j)}}k_w,k_w\rangle =\lim_{j\to\infty}\langle AU_{z^{(j)}}k_w,U_{z^{(j)}}k_w\rangle=\lim_{j\to\infty}\langle Ak_{\varphi_{z^{(j)}}(w)},k_{\varphi_{z^{(j)}}(w)}\rangle,
\]
where the last equation follows by (\ref{yousuanzi}). By (\ref{*}), we have $\lim_{j\to\infty}\langle Ak_{\varphi_{z^{(j)}}(w)},k_{\varphi_{z^{(j)}}(w)}\rangle=0$, so the Berezin transform of $\widetilde{A}$ equals to 0.

On the other hand, we have
$$
|\langle Ak_{z^{(j)}},k_{w^{(j)}}\rangle|=|\langle AU_{z^{(j)}}1,U_{z^{(j)}}k_{\varphi_{z^{(j)}}(w^{(j)})}\rangle|=|\langle U_{z^{(j)}}^*AU_{z^{(j)}}1,k_{\varphi_{z^{(j)}}(w^{(j)})}\rangle|.
$$
According to condition (1) and the M\"obius invariance of $\beta$, we have
\[
d(\varphi_{z^{(j)}}(w^{(j)}),0)=d(\varphi_{z^{(j)}}(w^{(j)}),\varphi_{z^{(j)}}(z^{(j)}))=d(w^{(j)},z^{(j)})\leq r.
\]
By passing to subsequence if necessary, we may assume that there exists $u\in\overline{D(0,r)}$ such that $\{\varphi_{z^{(j)}}(w^{(j)})\}$ converges to $u$ in norm as $j$ goes to $\infty$. Therefore,
\[
\lim_{j\to\infty}|\langle Ak_{z^{(j)}},k_{w^{(j)}}\rangle|=\lim_{j\to\infty}|\langle U_{z^{(j)}}^*AU_{z^{(j)}}1,k_{\varphi_{z^{(j)}}(w^{(j)})}\rangle|=|\langle \widetilde{A}1,k_u\rangle|=0,
\]
which  contradicts with (2).
\end{proof}

\noindent
{\bf Proof of Theorem A:} A direct computation shows that $k_z$ converges weakly to $0$ as $z$ goes to $\partial \mathbb{D}$. Then we just need to prove the `only if' part. By Proposition \ref{Proposition B} and linearity, we only need to prove that for each $X,Y\in\mathcal{D}_{0,0}$, $XAY$ is a compact operator.

We can assume that $X=\sum_{z\in\Gamma}c_zk_z\otimes k_z$, $Y=\sum_{z\in G}d_zk_z\otimes k_z$, where $\Gamma, G$ are $1-$separated and $\sup_{z\in\Gamma}|c_z|<\infty$, $\sup_{z\in G}|d_z|<\infty$.



For each pair of $\hat{k}=(k_1,\cdots,k_n)$ and $\hat{\nu}=(\nu_1,\cdots,\nu_n)$, where $k_i\in\{0, 1, 2,\cdots,5\}$ and $\nu_{i}\in\{1,\cdots,\widetilde{N}_0\}$ for $1\leq i\leq n$. By (\ref{quanbing}) and (\ref{yuji}), it is easy to check that there is a partition of
$$\Gamma=(\mathop\bigsqcup \limits_{\hat{\nu},\hat{k}}\Gamma^{(\hat{\nu},\hat{k})})\sqcup\Gamma_0,$$
where
\[
\Gamma^{(\hat{\nu},\hat{k})}\subset (\mathop\bigcup \limits_{\hat{\omega} \in \hat{I}_m^{(\hat{\nu},\hat{k})}}\hat{A}_{\hat{\omega}})\cap \Gamma \quad \text{and} \quad \Gamma_0=\Gamma\cap \{z\in\mathbb{D}:|z_i|<\rho_{3m}\,\text{ for any }\,1\le i \le n\}.
\]
And for each pair of $\hat{k}=(k_1,\cdots,k_n)$ and $\hat{\nu}=(\nu_1,\cdots,\nu_n)$, there also exists a partition
$$\Gamma^{(\hat{v},\hat{k})}=\bigsqcup_{\hat{j}\in\mathbb{N}^n}\bigsqcup_{\hat{u}\in E_{m,6\hat{j}+\hat{k}}^{(\hat{\nu})}}\Gamma_{\hat{j},\hat{u}}^{(\hat{\nu},\hat{k})},$$ where $\Gamma_{\hat{j},\hat{u}}^{(\hat{\nu},\hat{k})} \subset \Gamma^{(\hat{\nu},\hat{k})}\cap \hat{A}_{m,6\hat{j}+\hat{k},\hat{u}}$ and $E_{m,6\hat{j}+\hat{k}}^{(\hat{\nu})}=E_{m,6j_1+k_1}^{(\nu_1)}\times\dots \times E_{m,6j_n+k_n}^{(\nu_n)}$.  Define $X_{\hat{j},\hat{u}}^{(\hat{\nu},\hat{k})}=\sum_{z\in\Gamma_{\hat{j},\hat{u}}^{(\hat{\nu},\hat{k})}}c_zk_z\otimes k_z$. Then we have
$$ X=\sum_{\hat{\nu}\in \{1,\cdots ,\widetilde{N}_0\}^n}\sum_{\hat{k}\in\{0,\cdots,5\}^n}\sum_{\hat{j}\in\mathbb{N}^n}\sum_{\hat{u}\in E_{m,6\hat{j}+\hat{k}}^{(\hat{\nu})} }X_{\hat{j},\hat{u}}^{(\hat{\nu},\hat{k})}+X_0$$
where
$X_0=\sum_{z\in\Gamma_0}c_zk_z\otimes k_z$. So we have $XA=(\sum_{(\hat{\nu},\hat{k})}\sum_{\hat{j},\hat{u}}X_{\hat{j},\hat{u}}^{(\hat{\nu},\hat{k})}+X_0)A$.

For each $\hat{\nu},\hat{k},\hat{j}$ and $\hat{u}\in E_{m,6\hat{j}+\hat{k}}^{(\hat{\nu})}$, define $G^{(\hat{\nu},\hat{k})}=\bigcup_{\hat{j}\in\mathbb{N}^n}\bigcup_{\hat{u}\in E_{m,6\hat{j}+\hat{k}}^{(\hat{\nu})}}\hat{B}_{m,6\hat{j}+\hat{k},\hat{u}} \cap G.$ We also can easily check that there is a partition
\[
G^{(\hat{\nu},\hat{k})}=\bigsqcup_{\hat{j}\in\mathbb{N}^n}\bigsqcup_{\hat{u}\in E_{m,6\hat{j}+\hat{k}}^{(\hat{\nu})}} G_{\hat{j},\hat{u}}^{(\hat{\nu},\hat{k})},
\]
where $G_{\hat{j},\hat{u}}^{(\hat{\nu},\hat{k})}\subseteq G^{(\hat{\nu},\hat{k})}\cap \hat{B}_{m,6\hat{j}+\hat{k},\hat{u}}.$
Define
$$
Y^{(\hat{\nu},\hat{k})}=\sum_{{\hat{j}\in\mathbb{N}^n},\hat{u}\in E_{m,6\hat{j}+\hat{k}}^{(\hat{\nu})}}Y_{\hat{j},\hat{u}}^{(\hat{\nu},\hat{k})},
$$
where $Y_{\hat{j},\hat{u}}^{(\hat{\nu},\hat{k})}=\sum_{z\in G_{\hat{j},\hat{u}}^{(\hat{\nu},\hat{k})}}d_zk_z\otimes k_z$.
Then
\[
Y=Y^{(\hat{\nu},\hat{k})}+Y_{(\hat{\nu},\hat{k})},\text{ where }G_{(\hat{\nu},\hat{k})}=G\setminus G^{(\hat{\nu},\hat{k})}\text{ and }Y_{(\hat{\nu},\hat{k})}=\sum_{z\in G_{(\hat{\nu},\hat{k})}}d_zk_z\otimes k_z.
\]
Since card$\{(\hat{\nu},\hat{k}):\hat{\nu}\in \{1,2,\cdots,\widetilde{N}_0\}^n,\hat{k}\in\{0,\cdots,5\}^n,\text{ for any } 1\leq i\leq n\}$ is finite, it suffices to show that for each $(\hat{\nu},\hat{k})$, $X^{(\hat{\nu},\hat{k})}AY=K_m+S_m$, where $K_m$ is compact and $\|S_m\|$ is small when $m$ is large and $X_0AY$ is compact.

For $X_0AY$, by Lemma \ref{Lemma 3'}, we have that $\hat{A}_{m,\hat{0},\hat{0}}\subseteq D(z_{m,\hat{0},\hat{0}},L_m),$ where $\hat{A}_{m,\hat{0},\hat{0}}$ appears in (\ref{yuji}). Combining this with Lemma \ref{Lemma 0.1} (a), we have that $\Gamma_0$ is finite. Then $X_0AY$ is a compact operator.

For $X^{(\hat{\nu},\hat{k})}AY$, let
\begin{eqnarray*}
X^{(\hat{\nu},\hat{k})}AY &=& X^{(\hat{\nu},\hat{k})}AY^{(\hat{\nu},\hat{k})}+X^{(\hat{\nu},\hat{k})}AY_{(\hat{\nu},\hat{k})} \\
                          &=& \sum_{\hat{j},\hat{u}}X_{\hat{j},\hat{u}}^{(\hat{\nu},\hat{k})}AY_{\hat{j},\hat{u}}^{(\hat{\nu},\hat{k})}+\sum_{(\hat{j},\hat{u})\neq (\hat{j}',\hat{u}')}X_{\hat{j},\hat{u}}^{(\hat{\nu},\hat{k})}AY_{\hat{j}',\hat{u}'}^{(\hat{\nu},\hat{k})}+X^{(\hat{\nu},\hat{k})}AY_{(\hat{\nu},\hat{k})} \\
                          &=& T_{m,1}+T_{m,2}+T_{m,3},
\end{eqnarray*}
where $T_{m,1}=\sum_{\hat{j},\hat{u}}X_{\hat{j},\hat{u}}^{(\hat{\nu},\hat{k})}AY_{\hat{j},\hat{u}}^{(\hat{\nu},\hat{k})}$, $T_{m,2}=\sum_{(\hat{j},\hat{u})\neq (\hat{j}',\hat{u}')}X_{\hat{j},u}^{(\hat{\nu},\hat{k})}AY_{\hat{j}',\hat{u}'}^{(\hat{\nu},\hat{k})}$, $T_{m,3}=X^{(\hat{\nu},\hat{k})}AY_{(\hat{\nu},\hat{k})}$.

(1)  To prove that $T_{m,1}$ is a compact operator.

Since  $d(z,w)\leq 2L_m$ for $z\in \Gamma_{\hat{j},\hat{u}}^{(\hat{\nu},\hat{k})},w\in G_{\hat{j},\hat{u}}^{(\hat{\nu},\hat{k})}$, then  by the similar proof with $T_{R,1}$ in the proof of Proposition \ref{Proposition D},
\[
T_{m,1}=\sum_{\hat{j},\hat{u}}\sum_{z\in \Gamma_{\hat{j},\hat{u}}^{(\hat{\nu},\hat{k})},w\in G_{\hat{j},\hat{u}}^{(\hat{\nu},\hat{k})}}c_zd_w\langle Ak_w,k_z\rangle k_z\otimes k_w\in \text{ Span}(\mathcal{D}_0).
\]
 By Lemma \ref{Lemma F}, we have $|\langle Ak_w,k_z\rangle|\to 0$ as $z\to\partial \mathbb{D}$. Combining with Lemma \ref{corollary 2}, we have that $T_{m,1}$ is compact operator.


(2)  To estimate $\|T_{m,2}\|,\|T_{m,3}\|$.

Without loss of generality, let $\{e_z:z\in \Gamma\}$ and $\{\eta_z:z\in G\}$ be orthonormal sets with $\langle e_z, \eta_w\rangle=0$ for any $z\in \Gamma$, $w\in G$. We can decompose $T_{m,2}$ as
\[
T_{m,2}=\widetilde{X}\widetilde{T}_{m,2}\widetilde{Y}\text{, where }\widetilde{X}=\sum_{z\in \Gamma}c_zk_z\otimes e_z,\widetilde{Y}=\sum_{z\in G}d_z\eta_z\otimes k_z
\]

\[
\widetilde{T}_{m,2}=\sum_{(\hat{j},\hat{u})\neq (\hat{j}',\hat{u}')}\widetilde{X}_{\hat{j},\hat{u}}^{(\hat{\nu},\hat{k})}A\widetilde{Y}_{\hat{j}',\hat{u}'}^{(\hat{\nu},\hat{k})}\text{, where }\widetilde{X}_{\hat{j},\hat{u}}^{(\hat{\nu},\hat{k})}=\sum_{z\in \Gamma_{\hat{j},\hat{u}}^{(\hat{\nu},\hat{k})}}e_z\otimes k_z,~\widetilde{Y}_{\hat{j}',\hat{u}'}^{(\hat{\nu},\hat{k})}=\sum_{w\in G_{\hat{j}',\hat{u}'}^{(\hat{\nu},\hat{k})}}k_w\otimes \eta_w.
\]

By Proposition \ref{Proposition A}, it suffices to estimate $\|\widetilde{T}_{m,2}\|$.

For each $J>0$, define $N_J=\{\hat{j}\in\mathbb{N}^n:|j_i|\leq J\text{ for any }\, 1\le i\le n\}$ and
\[
\widetilde{T}_{m,2,J}=\sum_{(\hat{j},\hat{u})\neq (\hat{j}',\hat{u}'),\hat{j},\hat{j}'\in N_J, u\in E_{m,6\hat{j}+\hat{k}}^{(\hat{\nu})},\hat{u}'\in E_{m,6\hat{j'}+\hat{k}}^{(\hat{\nu})}}\widetilde{X}_{\hat{j},\hat{u}}^{(\hat{\nu},\hat{k})}A\widetilde{Y}_{\hat{j}',\hat{u}'}^{(\hat{\nu},\hat{k})},
\]
then we have $\widetilde{T}_{m,2,J}\to \widetilde{T}_{m,2}$ strongly as $J\to \infty$. Therefore, there is a $J=J(m)$ such that
\[
\|\widetilde{T}_{m,2}\|\leq2\|\widetilde{T}_{m,2,J}\|.
\]
 Combining this with Lemma \ref{KeyLemma}, there exists a subset $L_J\subset N_J$ such that
\begin{equation}\label{G2}
  \|\widetilde{T}_{m,2,J}\|\leq 4(\|(\sum_{\hat{j}\in L_J}\sum_{\hat{u}}\widetilde{X}_{\hat{j},\hat{u}}^{(\hat{\nu},\hat{k})})A(\sum_{\hat{j}'\in N_J\setminus L_J}\sum_{\hat{u}'}\widetilde{X}_{\hat{j}',\hat{u}'}^{(\hat{\nu},\hat{k})})\|+\|(\sum_{\hat{j}\in N_J\setminus L_J}\sum_{\hat{u}}\widetilde{X}_{\hat{j},\hat{u}}^{(\hat{\nu},\hat{k})})A(\sum_{\hat{j}'\in L_J}\sum_{\hat{u}'}\widetilde{Y}_{\hat{j}',\hat{u}'}^{(\hat{\nu},\hat{k})})\|)
\end{equation}
in which
\[
\sum_{\hat{j}\in L_J}\sum_{\hat{u}}\widetilde{X}_{\hat{j},\hat{u}}^{(\hat{\nu},\hat{k})}=\sum_{\hat{j}\in L_J}\sum_{\hat{u}}\sum_{z\in \Gamma_{\hat{j},\hat{u}}^{(\hat{\nu},\hat{k})}}e_z\otimes k_z,
\]
\[
\sum_{\hat{j}'\in N_J\setminus L_J}\sum_{\hat{u}'}\widetilde{X}_{\hat{j}',\hat{u}'}^{(\hat{\nu},\hat{k})}=\sum_{\hat{j}'\in N_J\setminus L_J}\sum_{\hat{u}'}\sum_{w\in G_{\hat{j}',\hat{u}'}^{(\hat{\nu},\hat{k})}}k_w\otimes \eta_w.
\]
If $z\in\Gamma_{\hat{j},\hat{u}}^{(\hat{\nu},\hat{k})}\subset \hat{A}_{m,6\hat{j}+\hat{k},\hat{u}}$, $w\in G_{\hat{j}',\hat{u}'}^{(\hat{\nu},\hat{k})}\subset \hat{B}_{m,6\hat{j}'+\hat{k}',\hat{u}'}$ and $(\hat{j},\hat{u})\neq(\hat{j}',\hat{u}')$, we have $d(z,w)\geq 2 \log m$ by Lemma \ref{Lemma 2} (c). The Lemma \ref{Lemma E} combined with (\ref{G2}), it gives that $\|\widetilde{T}_{m,2,J}\|$ is small when $m$ is large. Furthermore, $\|\widetilde{T}_{m,2}\|$ is small when $m$ is large. The estimation of $\|T_{m,3}\|$ is similar to $\|T_{m,2}\|$. Then we complete the proof.
$\Box$

\begin{proposition}\label{Proposition H}
 If $f\in L^{\infty}(\mathbb{D})$, then $T_f\in C^*(\mathcal{D})$.
\end{proposition}
\begin{proof}
By Proposition \ref{Proposition B} and linearity, we only need to prove that for any $X,Y\in\mathcal{D}_{0,0}$, $XT_fY\in\mathcal{D}$ for any $f\in L^{\infty}(\mathbb{D})$.

We can assume
$$X=\sum_{z\in \Gamma}c_zk_z\otimes k_z,~~Y=\sum_{w\in G}d_wk_w\otimes k_w,$$
where $\Gamma, G$ are $1-$separated and $\sup_{z\in\Gamma}|c_z|<\infty$, $\sup_{w\in G}|d_w|<\infty$.

Since $T_f=PM_fP$, where $P$ is the orhtogonal projection onto $L_a^2(\mathbb{D})$ and $M_f$ is the multiplication operator, then
\[
XT_fY=\sum_{(z,w)\in \Gamma\times G}c_zd_w\langle M_fk_w,k_z\rangle k_z\otimes k_w.
\]

For $R>0$, we can write
$$XT_fY=T_{R,1}+T_{R,2},$$
where
\[
T_{R,1}=\sum_{(z,w)\in \Gamma\times G,~d(z,w)<R}c_zd_w\langle M_fk_w,k_z\rangle k_z\otimes k_w,~~ T_{R,2}=\sum_{(z,w)\in \Gamma\times G,~d(z,w)\geq R}c_zd_w\langle M_fk_w,k_z\rangle k_z\otimes k_w.
\]
Using the similar method of the proof for $T_{R,1}\in $ Span$(\mathcal{D}_0)$ in the proof of Proposition \ref{Proposition D}, we have $T_{R,1}\in $ Span$(\mathcal{D}_0)$. So it suffices to prove that $\|T_{R,2}\|$ converges to 0 as $R\to\infty$.

The operator $T_{R,2}$ can be decomposed as
$$T_{R,2}=\widetilde{X}\widetilde{T}_{R,2}\widetilde{Y},$$
where
\[
\widetilde{X}=\sum_{z\in\Gamma}c_zk_z\otimes e_z, ~~~\widetilde{T}_{R,2}=\sum_{(z,w)\in \Gamma\times G,d(z,w)\geq R}\langle M_fk_z,k_w\rangle e_z\otimes u_w,~~~\widetilde{Y}=\sum_{w\in G}d_wu_w\otimes k_w.
\]
and $\{e_z:z\in \Gamma\}$, $\{u_w:w\in G\}$ are orthonormal sets.

By Proposition \ref{Proposition A}, we have
$$\|\widetilde{X}\|\leq C\sup_{z\in \Gamma}|c_z|,~~\|\widetilde{Y}\|\leq C\sup_{z\in G}|d_z|,\text{ for some }C>0.$$
For $\widetilde{T}_{R,2}$,
\begin{eqnarray*}
|\langle M_fk_z,k_w\rangle|\leq \|f\|_{\infty}|\langle |k_z|,|k_w|\rangle| &=& \|f\|_{\infty}\prod_{i=1}^n\int_D\frac{|k_{z_i}(u_i)|}{|k_{w_i}(u_i)|}|k_{w_i}(u_i)|^2dA_i(u_i) \\
     &=& \|f\|_{\infty}\prod_{i=1}^n\int_D\frac{|k_{z_i}(u_i)|}{|k_{w_i}(u_i)|}dA_i(\varphi_{w_i}(u_i)) \\
     &=& \|f\|_{\infty}\prod_{i=1}^n\int_D\frac{|k_{z_i}(\varphi_{w_i}(u_i))|}{|k_{w_i}(\varphi_{w_i}(u_i))|}dA_i(u_i) \\
     &=& \|f\|_{\infty} \prod_{i=1}^n \int_D|k_{\varphi_{w_i}(z_i)}(u_i)|dA_i(u_i),
\end{eqnarray*}
where the last equality follows from the fact that for any $u,v,\xi$ in $D$,
$k_u(\varphi_v(\xi))=\frac{k_{\varphi_v(u)}(\xi)}{k_v(\xi)}(\frac{|1-\overline{u}v|}{1-\overline{u}v})^2$. By Rudin-Forelli estimates, there exists $C_{\varepsilon}>0$ and $0<\varepsilon<\frac12$,
\begin{eqnarray*}
\prod_{i=1}^n\int_D|k_{\varphi_{w_i}(z_i)}(u_i)|dA_i(u_i)&\leq& \prod_{i=1}^n\int_{D}\frac{1-|\varphi_{w_i}(z_i)|^2}{|1-\overline{\varphi_{w_i}(z_i)}u_i|^2(1-|u_i|^2)^{\varepsilon}}  dA_i(u_i) \\
 &\leq & C_{\varepsilon}\prod_{i=1}^n(1-|\varphi_{w_i}(z_i)|)^{1-\varepsilon}=C_{\varepsilon}|\langle k_z,k_w\rangle|^{1-\varepsilon},
\end{eqnarray*}
which completes the proof by Lemma \ref{Lemma C}.
\end{proof}

Let $\mathcal{T}^{(1)}$ denote the $||\cdot ||$-closure of $\{T_f:f\in L^{\infty}(\mathbb{D})\}$.

\begin{proposition}\label{Proposition I}
$\mathcal{T}^{(1)}\supset \mathcal{D}_{0,0}$.
\end{proposition}
\begin{proof}
For $A\in\mathcal{D}_{0,0}$, we can assume that $A=\sum_{z\in\Gamma}c_zk_z\otimes k_z$, where $\Gamma$ is $\delta-$separated for some $\delta>0$ and $\sup_{z\in\Gamma}|c_z|<\infty$. For $r\in (0,\frac{\delta}{n})$, we define
\[
f_r=\sum_{z\in \Gamma}\frac{c_z}{\Lambda(\widetilde{D}(z,r))}\chi_{\widetilde{D}(z,r)}\in L^{\infty}(\mathbb{D}).
\]
Then
\begin{eqnarray*}
T_{f_r}-A &=& \int_{\mathbb{D}}f_r(u)k_u\otimes k_ud\Lambda(u)-\sum_{z\in \Gamma}c_zk_z\otimes k_z \\
          &=& \sum_{z\in \Gamma}\frac{c_z}{\Lambda(\widetilde{D}(z,r))}\int_{\widetilde{D}(z,r)}(k_u\otimes k_u-k_z\otimes k_z)d\Lambda(u) \\
          &=& \sum_{z\in \Gamma}\frac{c_z}{\Lambda(\widetilde{D}(0,r))}\int_{\widetilde{D}(0,r)}(k_{\varphi_z(u)}\otimes k_{\varphi_z(u)}-k_z\otimes k_z)d\Lambda(u) \\
          &=& \frac{1}{\Lambda(\widetilde{D}(0,r))} \int_{\widetilde{D}(0,r)}(G_u-G_0)d\Lambda(u),
\end{eqnarray*}
where $G_u=\sum_{z\in\Gamma}c_zk_{\varphi_z(u)}\otimes k_{\varphi_z(u)}=\sum_{z\in\Gamma}c_zU_zk_u\otimes U_zk_u$ for any $u\in \widetilde{D}(0,r)$.

If the map $u\mapsto G_u$ is norm continuous, then $\|T_{f_r}-A\|\to 0$ as $r\to 0$. Since $T_{f_r}\in \mathcal{T}^{(1)}$, then the proof is completed if we can prove $u\mapsto G_u$ is norm continuous. Let $\{e_z:z\in \Gamma\}$ be an orthonormal set. Since $\sup_{z\in\Gamma}|c_z|<\infty$, then it suffices to prove the map $u\mapsto H_u$ is continuous for $u\in \widetilde{D}(0,r)$, in which $H_u=\sum_{z\in\Gamma}U_zk_u\otimes e_z$. By Lemma \ref{Lemma 1}, there exists $C>0$ such that, for any $u,u'\in \widetilde{D}(0,r)$,
\[
\|H_u-H_{u'}\|\leq C\sup_{z\in\Gamma}\|U_zk_u-U_zk_{u'}\|_{\infty}=C\|k_u-k_{u'}\|_{\infty},
\]
which completes the proof.

\end{proof}

Recall that for each $z=(z_1,\dots,z_n)\in \mathbb{D}$, the formula
$$
K_z(\xi)=\prod_{i=1}^n\frac{1}{(1-\xi_i\bar{z_i})^2},\text{ for any }\xi=(\xi_1,\dots,\xi_n)\in \mathbb{D}$$
defines the reproducing kernel for $L_a^2(\mathbb{D})$. For each pair of $\alpha=(\alpha_1,\dots,\alpha_n)\in \mathbb{N}^n$ and $z\in \mathbb{D}$, we define
\begin{equation}\label{Kzalpha}
K_{z,\alpha}(\xi)=\prod_{i=1}^n\frac{\xi_i^{\alpha_i}}{(1-\xi_i\bar{z_i})^{2+\alpha_i}} \text{ for any }\xi=(\xi_1,\dots,\xi_n)\in \mathbb{D}.
\end{equation}

Note that $K_z=K_{z,(0,\cdots,0)}$ for every $z\in \mathbb{D}$.

\begin{proposition}\label{Proposition 1}
Let $\Gamma$ be a separated set in $\mathbb{D}$ and suppose that $\{c_u\}_{u\in\Gamma}$ is a bounded sequence. Then for any $\alpha\in\mathbb{N}^n$ and $z\in\mathbb{D}$, we have
\[
\sum_{u\in\Gamma}c_u(U_uK_z)\otimes (U_uK_{z,\alpha})\in \mathcal{T}^{(1)}.
\]
\end{proposition}

\begin{proof}
We prove this by an induction on $|\alpha|\doteq \sum_{i=1}^n\alpha_i$. Let $\alpha\in\mathbb{N}^n$ and $z\in\mathbb{D}$ be given.

For $|\alpha|=0$, by (\ref{yousuanzi2}), then we have
\[
\sum_{u\in\Gamma}c_u(U_uK_z)\otimes (U_uK_{z,(0,\cdots,0)})=\sum_{u\in\Gamma}c_u ||K_z||^2k_{\varphi_u(z)}\otimes k_{\varphi_u(z)}.
\]
By Lemma \ref{Lemma 0.1} (c), there is a partition $\Gamma=\Gamma_1\sqcup\Gamma_2\sqcup\cdots\sqcup\Gamma_m$ such that for each $i\in \{1,\dots,m\}$, the set $\{\varphi_u(z):u\in\Gamma_i\}$ is separated.
Therefore,
\[
\sum_{u\in\Gamma}c_u(U_uK_z)\otimes (U_uK_{z,(0,\cdots,0)})=\sum_{u\in\Gamma}c_u ||K_z||^2k_{\varphi_u(z)}\otimes k_{\varphi_u(z)}
\]
\[
=\sum_{j=1}^m\sum_{u\in\Gamma_j}c_u ||K_z||^2k_{\varphi_u(z)}\otimes k_{\varphi_u(z)}\in \text{ Span }(\mathcal{D}_{00}).
\]
Combined with Proposition \ref{Proposition I}, the proposition holds for $|\alpha|=0$.

For the inductive step, we suppose that for $k\in \mathbb{N}$, the proposition holds for every $\alpha\in \mathbb{N}^n$ with $|\alpha|\leq k$. Now for $\alpha \in \mathbb{N}^n$ with $|\alpha|=k+1$, we can decompose $\alpha$ as $\alpha=a+e^{(j)}$, where $|a|=k$ and $e^{(j)}=(0,\cdots,0,1,0,\cdots,0)$ with $1$ in $j$-th coordinate.

By induction hypothesis,
\begin{equation}\label{1.1}
\sum_{u\in\Gamma}c_u(U_uK_z)\otimes (U_uK_{z,a})\in \mathcal{T}^{(1)}\quad \text{for any}\, z\in\mathbb{D}.
\end{equation}

Now for fixed $z\in\mathbb{D}$, there is an $\varepsilon=\varepsilon(z)>0$ such that $z+c\in \mathbb{D}$ for any $c\in \mathbb{D}$ with $|c|\le \varepsilon.$
Therefore, for any $t\in [0,\varepsilon]$ and any $j\in\{1,\cdots,n\}$, we can define
\[
A_t^{(j)}=\sum_{u\in\Gamma}c_u(U_uK_{z+te^{(j)}})\otimes(U_uK_{z+te^{(j)},a})
\text{ and }
B_t^{(j)}=\sum_{u\in\Gamma}c_u(U_uK_{z+ite^{(j)}})\otimes(U_uK_{z+ite^{(j)},a}).
\]
By (\ref{1.1}), we have $A_t^{(j)},B_t^{(j)}\in \mathcal{T}^{(1)}$ for any $t\in [0,\varepsilon]$ and any $j\in\{1,\cdots,n\}$. Apparently, we have the following decomposition :  $$\sum_{u\in\Gamma}c_u(U_uK_z)\otimes (U_uK_{z,\alpha})=\frac{1}{2(2+a_j)}(X^{(j)}+Y^{(j)})\in\mathcal{T}^{(1)},$$
where
\[
X^{(j)}=\sum_{u\in\Gamma}c_u\{(2+a_j)(U_uK_z)\otimes(U_uK_{z,\alpha})+2(U_uK_{z,e^{(j)}})\otimes(U_uK_{z,a})\},
\]
\[
Y^{(j)}=\sum_{u\in\Gamma}c_u\{(2+a_j)(U_uK_z)\otimes(U_uK_{z,\alpha})-2(U_uK_{z,e^{(j)}})\otimes(U_uK_{z,a})\}.
\]
Then it suffices to prove that
\begin{equation}\label{1.2}
\lim_{t\to 0}\|t^{-1}(A_t^{(j)}-A_0^{(j)})-X^{(j)}\| = 0
\end{equation}
 and
\begin{equation}\label{1.3} \lim_{t\to 0}\|(it)^{-1}(B_t^{(j)}-B_0^{(j)})-Y^{(j)}\| = 0.
\end{equation}

In order to prove (\ref{1.2}), we decompose $t^{-1}(A_t^{(j)}-A_0^{(j)})$ and $ X^{(j)}$ as follows:
\[
t^{-1}(A_t^{(j)}-A_0^{(j)})=G_t^{(j)}+H_t^{(j)}\text{ and }X^{(j)}=V^{(j)}+W^{(j)},
\]
in which
\[
G_t^{(j)}=\frac{1}{t}\sum_{u\in\Gamma}c_u\{U_u(K_{z+te^{(j)}}-K_z)\otimes U_zK_{z,a}\}\text{ and }H_t^{(j)}=\frac{1}{t}\sum_{u\in\Gamma}(U_uK_{z+te^{(j)}})\otimes\{U_u(K_{z+te^{(j)},a}-K_{z,a})\},
\]
\[
V^{(j)}=\sum_{u\in\Gamma}c_u(2+a_j)(U_uK_z)\otimes (U_zK_{z,\alpha})\text{ and }W^{(j)}=\sum_{u\in\Gamma}2c_u(U_uK_{z,e^{(j)}})\otimes(U_uK_{z,a}).
\]
Since $\|t^{-1}(A_t^{(j)}-A_0^{(j)})-X^{(j)}\|\leq \|H_t^{(j)}-V^{(j)}\|+\|G_t^{(j)}-W^{(j)}\|$, then we only need to prove that $$\lim_{t\to 0}\|H_t^{(j)}-V^{(j)}\|=0 \quad\text{and}\quad \lim_{t\to 0}\|G_t^{(j)}-W^{(j)}\|=0.$$

For any $0<t\leq \varepsilon$ and any $j\in\{1,\cdots,n\}$, we can rewrite $H_t^{(j)}-V^{(j)} = S_t^{(j)}+ T_t^{(j)}$, where
\[
S_t^{(j)}=\sum_{u\in\Gamma}c_u(U_uK_{z+te^{(j)}})\otimes\{U_u(t^{-1}(K_{z+te^{(j)},a}-K_{z,a})-(2+a_j)K_{z,\alpha})\},
\]
\[
T_t^{(j)}=(2+a_j)\sum_{u\in\Gamma}c_u\{U_u(K_{z+te^{(j)}}-K_z)\}\otimes(U_uK_{z,\alpha}).
\]
Let  $\{e_u\}_{u\in\Gamma}$ be an orthonormal set. Then we have $S_t^{(j)}=\widetilde{S}_t^{(j)}(\widehat{S}_t^{(j)})^*$, in which
\[
\widetilde{S}_t^{(j)}=\sum_{u\in\Gamma}c_u(U_uK_{z+te^{(j)}})\otimes e_u\text{ and }
\widehat{S}_t^{(j)}=\sum_{u\in\Gamma}\{U_u(t^{-1}(K_{z+te^{(j)},a}-K_{z,a})-(2+a_j)K_{z,\alpha})\}\otimes e_u.
\]
For $C=\mathrm{sup}_{u\in \Gamma}|c_u|$, by Lemma \ref{Lemma 1}, there exists $C_1,C_2>0$ such that
$$
\|\widetilde{S}_t^{(j)}\|\leq CC_1\|K_{z+te^{(j)}}\|_{\infty}\text{ and }
\|\widehat{S}_t^{(j)}\|\leq C_2\|t^{-1}(K_{z+te^{(j)},a}-K_{z,a})-(2+a_j)K_{z,\alpha}\|_{\infty}.
$$
A direct computation shows that
\[
\lim_{t\to 0}\{t^{-1}(K_{z+te^{(j)},a}-K_{z,a})(\xi)-(2+a_j)K_{z,\alpha}(\xi)\}=0.
\]
then the convergence above is uniform for any $\xi\in\mathbb{D}$ with fixed $0\leq t\le \varepsilon$ and fixed $z$, i.e.
\begin{equation}\label{daoshugongshi}
\lim_{t\to 0}\|t^{-1}(K_{z+te^{(j)},a}-K_{z,a})-(2+a_j)K_{z,\alpha}\|_{\infty}=0.
\end{equation}
It follows that $S_t^{(j)}=\widetilde{S}_t^{(j)}(\widehat{S}_t^{(j)})^*$ tends to 0 as $t$ goes to 0. The case of $T_t^{(j)}$ is similar to $S_t^{(j)}$. We can write  $T_t^{(j)}=\widetilde{T}_t^{(j)}(\widehat{T}^{(j)})^*$,
where  $$\widetilde{T}_t^{(j)}=(2+a_j)\sum_{u\in\Gamma}c_u\{U_u(K_{z+te^{(j)}}-K_z)\}\otimes e_u \quad \text{and} \quad \widehat{T}^{(j)}=\sum_{u\in \Gamma}(U_uK_{z,\alpha})\otimes e_u.$$
By Lemma \ref{Lemma 1}, we have $\widehat{T}^{(j)}$ is a bounded operator and for some $C_3>0$,
$$ \|\widetilde{T}_t^{(j)}\|\leq C C_3(2+a_j)\|K_{z+te^{(j)}}-K_{z,a}\|_{\infty}.$$
For the given $z$ and $0\le t \le \varepsilon$, it is easy to check that $\lim_{t\to 0}\|K_{z+te^{(j)}}-K_z\|_{\infty}=0$, which implies $\lim_{t\to 0}\|T_t^{(j)}\|=0$. Therefore, we have $\lim_{t\to 0}\|H_t^{(j)}-V^{(j)}\|=0$. By the similar way, we also have $\lim_{t\to 0}\|G_t^{(j)}-W^{(j)}\|=0$.

The proof of (\ref{1.3}) uses essentially the same argument as above, and the only additional care that needs to be taken is the following: The rank-one operator $f\otimes g$ is linear with respect to $f$ and conjugate linear with respect to $g$. Moreover, the $\xi_i\bar{z_i}$ is conjugate linear with respect to $z_i$. These are the properties that determine the + and - signs in each term $c_u\{\cdots\}$ in the sum that defines the operator $Y^{(j)}$.
This completes the proof of the proposition.
\end{proof}

\begin{proposition}\label{Proposition 2}
Let $\Gamma$ be a separated set in $\mathbb{D}$ and suppose that $\{c_u\}_{u\in\Gamma}$ is a bounded sequence in $\mathbb{D}$. Then for any $w\in\mathbb{D}$, we have
\[
\sum_{u\in\Gamma}c_uk_u\otimes k_{\varphi_u(w)}\in \mathcal{T}^{(1)}.
\]
\end{proposition}

\begin{proof}
For any $\alpha=(\alpha_1,\cdots,\alpha_n)\in \mathbb{N}^n$, define
\[
P_{\alpha}(\xi)=\prod_{i=1}^n\xi_i^{\alpha_i},\text{ for any }\xi=(\xi_1,\dots,\xi_n)\in \mathbb{D}.
\]
By (\ref{Kzalpha}), it is easy to see that $P_{\alpha}(\xi)=K_{0,\alpha}(\xi)$.
For a fixed $u=(u_1,u_2,\dots,u_n)\in\mathbb{D}$, define the function
\begin{equation}\label{1.8}
d_u(w)=c_u\prod_{i=1}^n\left(\frac{1-w_i\bar{u}_i}{|1-w_i\bar{u}_i|}\right)^2 , \, \text{ for any }\,w=(w_1,\dots,w_n)\in \mathbb{D}.
\end{equation}
Note that $U_uK_0=U_u1=k_u$ for every $u\in \Gamma$ and $|c_u|=|d_u(w)|$ for any $w\in \mathbb{D}$.
Using Proposition \ref{Proposition 1}, we have
\begin{equation}\label{1.4}
\sum_{u\in\Gamma}d_u(w)k_u\otimes(U_uP_{\alpha})=\sum_{u\in\Gamma}c_u(U_uK_0)\otimes(U_uK_{0,\alpha})\in \mathcal{T}^{(1)}.
\end{equation}
 For each $i\in\{1,\cdots,m\}$, $j\in\mathbb{N}$ and each $w=(w_1,\dots,w_n)\in \mathbb{D}$, we define the function
 $$g^{(i)}_{w,j}(\xi)=(\xi_i\overline{w_i})^j \text{ for any }\xi=(\xi_1,\dots,\xi_n)\in \mathbb{D}.$$
It is easy to see $g^{(i)}_{w,j}\in$ Span$\{P_{\alpha}:\alpha\in\mathbb{N}^n\}$. For any $k\in\mathbb{N}$, let
\[
A_k=\sum_{u\in\Gamma}d_u(w)k_u\otimes (U_u(\prod_{i=1}^n(\sum_{j_i=0}^k(j_i+1)g_{w,j_i}^{(i)}))).
\]
Since $\prod_{i=1}^n(\sum_{j_i=0}^k(j_i+1)g_{w,j_i}^{(i)})\in$ Span$\{P_{\alpha}:\alpha\in\mathbb{N}^n\}$, then for every $k\in \mathbb{N}$, then, by (\ref{1.4}), we have $A_k\in\mathcal{T}^{(1)}$ for any $k\in\mathbb{N}$. Taking an orthnormal set $\{e_u:u\in\Gamma\}$, we have the factorization $A_k=TB_k^*$, where
\[
T=\sum_{u\in\Gamma}d_u(w)k_u\otimes e_u\text{ and }B_k=\sum_{u\in\Gamma}(U_u(\prod_{i=1}^n(\sum_{j_i=0}^k(j_i+1)g_{w,j_i}^{(i)})))\otimes e_u.
\]
Denote $G= \sum_{u\in\Gamma}(U_uK_w)\otimes e_u$. By Lemma \ref{Lemma 1}, we have that $T$ is a bounded operator and there exists $C>0$ such that
\begin{equation}\label{1.5}\|G-B_k\|\le C\|K_w-\prod_{i=1}^n(\sum_{j_i=0}^k(j_i+1)g_{w,j_i}^{(i)})\|_{\infty} \text{ for any }k\in\mathbb{N}.\end{equation}
Since for any $|c|<1$, $(1-c)^{-2}=\sum_{j=0}^{\infty}(1+j)c^j$, for the $w\in\mathbb{D}$ given above and any $z=(z_1,\dots,z_n)\in\mathbb{D}$, there exists $k_0>0$ such that for any $k\geq k_0$, $$|\sum_{j=0}^{k}(j_i+1)(z_i\overline{w_i})^{j_i}|< |\frac{1}{(1-z_i\overline{w}_i)^2}| +1, \, \text{ for any }\, 1\le i\le n.$$
Write $a=\prod_{i=1}^n(|\frac{1}{(1-z_i\overline{w}_i)^2}| +1)^2$. For the fixed $w\in \mathbb{D}$, by a telescope sum, we have
\begin{eqnarray*}
& &\lim_{k\rightarrow \infty}|\prod_{i=1}^n\frac{1}{(1-z_i\bar{w}_i)^2}-\prod_{i=1}^n(\sum_{j_i=0}^k(j_i+1)g_{w,j_i}^{(i)}(z))|\\
& \le& \lim_{k\rightarrow \infty} \sum_{i=1}^{n}a \cdot |\frac{1}{(1-z_i\overline{w_i})^2}-\sum_{j_i=0}^k(j_i+1)(z_i\overline{w_i})^{j_i}|=0.
\end{eqnarray*}
Since $w\in \mathbb{D}$ is fixed, the convergence above is uniform for any $z\in\mathbb{D}$. Thus, for fixed $w\in\mathbb{D}$,
 $$\lim_{k\to \infty}\|\prod_{i=1}^n\frac{1}{(1-z_i\bar{w}_i)^2}-\prod_{i=1}^n(\sum_{j_i=0}^k(j_i+1)g_{w,j_i}^{(i)}(z))\|_{\infty}=0.$$
Combined with (\ref{1.5}), it gives that
\[\lim_{k\to \infty}\|TG^*-A_k\|=\lim_{k\to \infty}
\|TG^*-TB_k^*\|\leq \lim_{k\to \infty}\|T\|\|G-B_k\|=0.
\]
Since each $A_k\in\mathcal{T}^{(1)}$, then we have
\begin{equation}\label{1.6}\sum_{u\in\Gamma}d_u(w)k_u\otimes (U_uK_w)=TG^*\in \mathcal{T}^{(1)}.\end{equation}
Combining with (\ref{yousuanzi}), we have
\begin{equation}\label{1.7}U_uk_w=\prod_{i=1}^n\left(\frac{|1-w_i\bar{u_i}|}{1-w_i\bar{u_i}}\right)^2k_{\varphi_u(w)}=\frac{c_u}{d_u(w)}k_{\varphi_u(w)}.\end{equation}
Combining this with $\frac{K_w}{\|K_w\|}=k_w$, we find that $$\sum_{u\in\Gamma}c_uk_u\otimes k_{\varphi_u(w)}=\sum_{u\in\Gamma}c_uk_u\otimes \frac{d_u(w)}{c_u} U_uk_w =\frac{1}{\|K_w\|}\sum_{u\in\Gamma}d_u(w) k_u\otimes U_uK_w\in \mathcal{T}^{(1)}$$ by (\ref{1.6}).
\end{proof}

\begin{proposition}\label{Proposition J}
$\mathcal{D}_0 \subseteq \mathcal{T}^{(1)}.$
\end{proposition}
\begin{proof}
For any $T\in \mathcal{D}_0$, we can assume that $$T=\sum_{u\in \Gamma}c_uk_u\otimes k_{\gamma(u)},$$
where $\Gamma$ is a separated set in $\mathbb{D}$, $\mathrm{sup}_{u\in \Gamma}|c_u|<\infty$ and $\gamma:\Gamma\rightarrow \mathbb{D}$ is any map for which there exists a $0\le C< \infty$ such that $d(z,\gamma(z))\le C$ for every $z\in \Gamma$.

Write $K=\{w\in \mathbb{D}:d(0,w)\le C\}$ and $\psi(u)=\varphi_u(\gamma(u))$ for any $u\in \Gamma.$ Note that $\varphi_u(\psi(u))=\gamma(u)$ and $\varphi_u(u)=0.$ By the M\"obius invariance of $\beta$, we have
$$d(0,\psi(u))=d(\varphi_u(u),\varphi_u(\gamma(u)))=d(u,\gamma(u))\le C.$$
That is, $\psi(u)\in K$ for any $u\in \Gamma$. By (\ref{1.7}) and the function (\ref{1.8}), we have
$$|d_u(\psi(u))|=|c_u|\text{ and }T=\sum_{u\in \Gamma}d_u(\psi(u))k_u\otimes (U_uk_{\psi(u)}).$$
Taking an orthonormal set $\{e_u:u\in\Gamma\}$, we have the factorization $T=AB^*$, where $$A=\sum_{u\in \Gamma}d_u(\psi(u))k_u\otimes e_u,\,\text{and}\,B=\sum_{u\in \Gamma} (U_uk_{\psi(u)})\otimes e_u.$$
Since $\sup_{u\in\Gamma}\{|d_u(\psi(u))|\}<\infty$, then $A$ is a bounded operator. Let $\varepsilon>0$ be given. Combined with  compactness of $K$ and the fact that the map $z\mapsto k_z$ is $\|\cdot\|_{\infty}$-continuous on $K$, there exists a finite open cover $\{\Omega_1,\dots,\Omega_m\}$ of $K$ and $z^{(i)}\in \Omega_i\cap K,$ for each $i\in\{1,\dots,m\}$ such that
$$\|k_{z^{(i)}}-k_w\|_{\infty}<\varepsilon,\text{ for any }w\in \Omega_i \text{ and any }i\in\{1,\dots,m\}.$$
Therefore, we obtain a partition
$$K=E_1\sqcup\dots\sqcup E_m$$
such that $E_i\subseteq \Omega_i$ for every $i\in \{1,\dots,m\}.$ Write $\Gamma_i=\{u\in \Gamma:\psi(u)\in E_i\}$, $i=1,\dots,m$. Then the set $\{\Gamma_1,\dots,\Gamma_m\}$ forms a partition of $\Gamma$ and $\|k_{z^{(i)}}-k_{\psi(u)}\|_{\infty}<\varepsilon$ for any $u\in \Gamma_i$. For every $i\in \{1,\dots,m\},$ we define the operator
$$B_i=\sum_{u\in \Gamma_i} (U_uk_{z^{(i)}})\otimes e_u.$$
Then we have
$$AB_i^*=\sum_{u\in \Gamma_i}d_u(\psi(u))k_u\otimes (U_uk_{z^{(i)}})=\sum_{u\in \Gamma_i}d_u(\psi(u))\left(\frac{c_u}{d_u(z^{(i)})} \right) k_u\otimes k_{\varphi_u(z^{(i)})}.$$
where the last equality follows by (\ref{yousuanzi}). Thus it follows from Proposition \ref{Proposition 2} that $AB_i^*\in \mathcal{T}^{(1)}$ for each $i\in \{1,\dots,m\}$. By Lemma \ref{Lemma 1} and the fact that $\Gamma_i\cap \Gamma_j=\emptyset$ for any $i,j\in \{1,\dots,m\}$, there exists $C_1>0$ such that
$$\|B-(B_1+\dots+B_m)\|=\|\sum_{i=1}^m\sum_{u\in \Gamma_i}\{U_u(k_{\psi(u)}-k_{z^{(i)}})\}\otimes e_u\|\le C_1\cdot\mathop{max}\limits_{1\le i \le m}\mathop{\mathrm{sup}}\limits_{u\in \Gamma_i}\|k_{\psi(u)}-k_{z^{(i)}}\|_{\infty}\le C_1\varepsilon.$$
Therefore
$$\|T-\sum_{i=1}^m AB^*_i\|\le \|A\|\cdot \|B^*-(B_1^*+\dots+B_m^*)\|= \|A\|\cdot\|B-(B_1+\dots+B_m)\|\le \|A\|\cdot C_1\varepsilon. $$
Since $\{AB_1^*,\dots,AB_m^*\}\subseteq \mathcal{T}^{(1)}$, then we have $T\in \mathcal{T}^{(1)}$.
\end{proof}
\
To sum up, we have the following localization of Toeplitz algebra on polydisk.\\

\noindent
{\bf Proof of Theorem B:}
By Proposition \ref{Proposition H} and Proposition \ref{Proposition D}, we have $\mathcal{T}(\mathbb{D})\subseteq C^*(\mathcal{D})=\mathcal{D}$. While Proposition \ref{Proposition J} ensures that $\mathcal{D}_0\subseteq \mathcal{T}^{(1)}\subseteq \mathcal{T}(\mathbb{D})$, so we have $C^*(\mathcal{D})=\mathcal{T}(\mathbb{D})$.
$\Box$
%
%
%
%
\section{Essential Commutant}
\

Following the ideas in \cite{J. Xia 2018}, we will now generalize the notion of localized operators to polydisk $\mathbb{D}$.

\begin{definition}\label{Definition 000}
For any bounded linear operator $A$ on $L_a^2(\mathbb{D})$, $\mathcal{D}_0(A)$ denotes the collection of operators of the form
$$\sum_{z\in \Gamma}c_z\langle Ak_{\psi(z)},k_{\varphi(z)}\rangle k_{\varphi(z)}\otimes k_{\psi(z)},$$
where $\Gamma$ is any separated set in $\mathbb{D}$, $\{c_z:z\in \Gamma\}$ is any bounded set of complex coefficients and $\varphi,\psi:\Gamma\to \mathbb{D}$ are any maps for which there is a $0\leq C<\infty$ such that
$$d(z,\varphi(z))\leq C,~~d(z,\psi(z))\leq C,$$
for every $z\in\Gamma$.

For any bounded linear operator $A$ on $L_a^2(\mathbb{D})$, $\mathcal{D}(A)$ denotes the closure of the linear span of $\mathcal{D}_0(A)$ with respect to the operator norm.
\end{definition}

\begin{definition}\label{Definition 00}
Let $A$ be a bounded operator on the Bergman space $L_a^2(\mathbb{D})$. We denote LOC$(A)$ be the collection of operators of the form
$$T=\sum_{z\in\Gamma}T_{f_z}AT_{f_z},$$
where $\Gamma$ is any separated set in $\mathbb{D}$ and $\{f_z: z\in \Gamma\}$ is any family of continuous functions
on $\mathbb{D}$ satisfying the following three conditions:

(1) There is a $0 < \rho < \infty$ such that $f_z = 0$ on $\mathbb{D}\setminus \widetilde{D}(z, \rho)$ for every $z \in\Gamma$.

(2) The inequality $0 \leq f_z \leq 1$ holds on $\mathbb{D}$ for every $z\in\Gamma$.

(3) The family $\{f_z : z\in\Gamma\}$ satisfies a uniform Lipschitz condition on $\mathbb{D}$ with respect to the metric $d$. That is, there is a positive number $C$ such that $|f_z(\zeta) -f_z(\xi)| \leq Cd(\zeta, \xi)$ for all
$z\in\Gamma$ and $\zeta,\xi\in \mathbb{D}$.
\end{definition}

\begin{proposition}\label{Proposition K}
For any $A\in B(L_a^2(\mathbb{D}))$, $LOC(A)\subset \mathcal{D}(A)$.
\end{proposition}

\begin{proof}
Recall that for any $f\in L^{\infty}(\mathbb{D})$, the Toeplitz operator $T_f$ can be represented as
$$T_f=\int_{\mathbb{D}}f(z)k_z\otimes k_zd\Lambda(z).$$
Let $T$ be any operator in $LOC(A)$. We assume
$$T=\sum_{z\in\Gamma}T_{f_z}AT_{f_z},$$
for some separated set $\Gamma\subset \mathbb{D}$ and and $\{f_z: z\in \Gamma\}$ is any family of continuous functions
on $\mathbb{D}$ satisfying three conditions which appear in Definition \ref{Definition 00}.
Therefore, for each $z\in\Gamma$,
\begin{eqnarray*}
  T_{f_z}AT_{f_z} &=& \int_{\mathbb{D}}f_z(u)k_u\otimes k_u d\Lambda(u) A \int_{\mathbb{D}}f_z(v)k_v\otimes k_v d\Lambda(v)\\
                  &=& \int_{\widetilde{D}(z,\rho)}f_z(u)k_u\otimes k_u d\Lambda(u) A \int_{\widetilde{D}(z,\rho)}f_z(v)k_v\otimes k_v d\Lambda(v) \\
                  &=& \int_{\widetilde{D}(0,\rho)}f_z(\varphi_z(u))k_{\varphi_z(u)}\otimes k_{\varphi_z(u)} d\Lambda(u) A \int_{\widetilde{D}(0,\rho)}f_z(\varphi_z(v))k_{\varphi_z(v)}\otimes k_{\varphi_z(v)} d\Lambda(v) \\
                  &=& \iint_{\widetilde{D}(0,\rho)\times \widetilde{D}(0,\rho)}f_z(\varphi_z(u))f_z(\varphi_z(v))\langle Ak_{\varphi_z(v)},k_{\varphi_z(u)}\rangle k_{\varphi_z(u)}\otimes k_{\varphi_z(v)} d\Lambda(u)d\Lambda(v),\\
\end{eqnarray*}
where the second equality follows by the conditiong (1) that appears in Definition \ref{Definition 00}. So
\[
T=\sum_{z\in\Gamma}T_{f_z}AT_{f_z}=\sum_{z\in\Gamma}\iint_{\widetilde{D}(0,\rho)\times \widetilde{D}(0,\rho)}f_z(\varphi_z(u))f_z(\varphi_z(v))\langle Ak_{\varphi_z(v)},k_{\varphi_z(u)}\rangle k_{\varphi_z(u)}\otimes k_{\varphi_z(v)} d\Lambda(u)d\Lambda(v).
\]

Since $\overline{\widetilde{D}(0,\rho)\times \widetilde{D}(0,\rho)}$ is compact, then for any $\eta>0$, there exists $m=m(\eta)\in\mathbb{N}_+$ and a finite partition  $\overline{\widetilde{D}(0,\rho)\times \widetilde{D}(0,\rho)}=E_1\sqcup\cdots\sqcup E_m$ such that every diam$(E_j)<\eta$ for any $1\leq j\leq m$. We can choose a point $(u^{(j)},v^{(j)})\in E_j$ for any $j$ and define
\[
X^{(j)}=(\Lambda\times\Lambda)(E_j)\sum_{z\in\Gamma}f_z(\varphi_z(u^{(j)}))f_z(\varphi_z(v^{(j)}))\langle Ak_{\varphi_z(v^{(j)})},k_{\varphi_z(u^{(j)})}\rangle k_{\varphi_z(u^{(j)})}\otimes k_{\varphi_z(v^{(j)})}.
\]
Since $(\Lambda\times\Lambda)(E_j)$, $f_z$ and $A$ are bounded as a number, a function and an operator respectively, by Definition \ref{Definition 000}, $X^{(j)}\in \mathcal{D}_0(A)$ for any $1\leq j\leq m$. Let $X=X^{(1)}+\cdots+X^{(m)}$. Then we have
\begin{eqnarray*}
T-X &=& \sum_{j=1}^m[\iint_{E_j}(\sum_{z\in\Gamma}(f_z(\varphi_z(u))f_z(\varphi_z(v))\langle Ak_{\varphi_z(v)},k_{\varphi_z(u)}\rangle k_{\varphi_z(u)}\otimes k_{\varphi_z(v)}d\Lambda(u)d\Lambda(v))\\
    & & -(\Lambda\times\Lambda)(E_j)(\sum_{z\in\Gamma}(f_z(\varphi_z(u^{(j)}))f_z(\varphi_z(v^{(j)}))\langle Ak_{\varphi_z(v^{(j)})},k_{\varphi_z(u^{(j)})}\rangle k_{\varphi_z(u^{(j)})}\otimes k_{\varphi_z(v^{(j)})})]\\
    &=& \sum_{j=1}^m\iint_{E_j}[(\sum_{z\in\Gamma}(f_z(\varphi_z(u))f_z(\varphi_z(v))\langle AU_zk_v,U_zk_u\rangle U_zk_u\otimes U_zk_v)\\
    & & -(\sum_{z\in\Gamma}(f_z(\varphi_z(u^{(j)}))f_z(\varphi_z(v^{(j)}))\langle AU_zk_{v^{(j)}},U_zk_{u^{(j)}}\rangle U_zk_{u^{(j)}}\otimes U_zk_{v^{(j)}})]d\Lambda(u)d\Lambda(v).
\end{eqnarray*}
Write
\begin{equation}\label{K1}
G{(u,v)}= \sum_{z\in\Gamma}(f_z(\varphi_z(u))f_z(\varphi_z(v))\langle AU_zk_v,U_zk_u\rangle U_zk_u\otimes U_zk_v,
\end{equation}
and $$A{(u,v)}= \sum_{z\in\Gamma}f_z(\varphi_z(u))\langle AU_zk_v,U_zk_u\rangle U_zk_u\otimes e_z,\quad B(v)^*= \sum_{z\in\Gamma}f_z(\varphi_z(v)) e_z\otimes U_zk_v, $$
where $\{e_z,z\in \Gamma\}$ is an orthonormal set. If $G(u,v)$ is norm continuous, then $\|T-X\|\to 0$ as diam$(E_j)\to 0$ for any $j\in\{1,\cdots,m\}$. Therefore, it suffices to prove that $G{(u,v)}$ in (\ref{K1}) is norm continuous.

By Lemma \ref{Lemma 1} and Lemma \ref{Lemma 0.1} (c), there exists $C>0$ such that $\|A(u,v)\|<C$ and $\|B\|<C$ for any $u,v\in \widetilde{D}(0,\rho)$.
Note that $G{(u,v)}=A{(u,v)}B(v)^*$. Then for any $u,u',v,v'\in \widetilde{D}(0,\rho)$, we have
$$\|G(u,v)-G(u',v')\|\le \|A(u,v)-A(u',v')\|\cdot\|B(v)\|+\|A(u',v')\|\cdot \|B(v)-B(v')\|.$$
So it suffices to prove that the maps $A{(u,v)}$ and $B(v)$ are continuous. It is obvious that the map $(u,v)\mapsto \langle k_u,k_v\rangle$ is norm continuous. Then combined with the formula
$$\langle AU_zk_u,U_zk_v\rangle-\langle AU_zk_{u'},U_zk_{v'}\rangle=\langle AU_z(k_u-k_{u'}),U_zk_v\rangle-\langle AU_zk_{u'},U_z(k_v-k_{v'})\rangle,$$ we have that the map
$$(u,v)\mapsto \langle AU_zk_u,U_zk_v\rangle$$
is norm continuous. Since $\overline{\widetilde{D}(0,\rho)}$ is compact, then there exists a $0<c<\infty$ such that $1-|u|\ge c$ for any $u\in \overline{\widetilde{D}(0,\rho)}$. Therefore,
$$\|k_u-k_{u'}\|_{\infty}\rightarrow 0 \quad \text{as}\quad d(u,u')\rightarrow 0,\text{ for any }u,u'\in \overline{\widetilde{D}(0,\rho)}.$$
By the condition (3) in Definition \ref{Definition 00}, we have $u \mapsto f_z(\varphi_z(u))$ is uniform continuous. Then we have that the map $(u,v)\mapsto A(u,v)$ and  $(u,v)\mapsto B(v)$ both are norm continuous.
\end{proof}

\section{Some Estimation}
\

In this section, we will prove that any operators in $EssCom\{T_g:g\in VO_{bdd}\}$ would satisfy  $\varepsilon-\delta$ condition.

\begin{definition}\label{banjing}
For each $0<t<1$, let $\mathbb{D}_t$ be the collection of $z=(z_1,\dots,z_n)\in \mathbb{D}$ satisfying the condition $|z_i|< t \, \text{ for every }\, i\in\{1,\cdots,n\}$.
\end{definition}

\begin{definition}\label{Definition 4}
\

(a)  For each $ 0<t<1$, the symbol $\Psi(t)$ denotes the collection of  $g$ in $C(\mathbb{D})$ (the collection of all continuous functions on $\mathbb{D}$) satisfying

~~~(1)  $0\leq g(z)\leq 1$, on $\mathbb{D}$;

~~~(2)  $g(z)=1$ whenever $z \in \overline{\mathbb{D}_t}$;

~~~(3)  There is a $t'=t'(g)\in (t,1)$ such that $g(z)=0$ whenever $z\in \mathbb{D}\backslash \mathbb{D}_{t'}$.

(b)  For $0<t<1$ and $\delta>0$, let $\Psi(t;\delta)=\{g\in \Psi(t): \mathrm{diff}(g)\leq \delta\}$.
\end{definition}

\begin{lemma}\label{Lemma 4}
For any $t\in(0,1)$ and any $\delta>0$, $\Psi(t;\delta)\neq\emptyset$.
\end{lemma}
\begin{proof}
Let $F:\mathbb{R}\to [0,1]$ be a function defined by
\[
F(x) =
\begin{cases}
1 &\quad\text{ if }x\leq 0\\
1-x &\quad\text{ if }x\in(0,1) \\
0 &\quad\text{ if }x\geq 1.\\
\end{cases}
\]
Then $|F(x)-F(y)|\leq |x-y|$ for all $x,y\in\mathbb{R}$.

Let $0<t<1$ and $\delta>0$ be given. Let $b=\delta\beta(t,0)$,  then the function defined by
$$g(z)=\prod_{i=1}^nF(\delta(z_i,0)-b),~~\text{ for any } z=(z_1,\cdots,z_n)\in\mathbb{D}$$
belongs to $\Psi(t;\delta)$.
\end{proof}

\begin{lemma}\label{Lemma 5}
For any $f\in L^{\infty}(\mathbb{D})$ and any $h\in L_a^2(\mathbb{D})$, we have
\begin{equation}\label{5.1}
\lim_{t\to 1}\sup\{\|T_{gf}h-T_fh\|:g\in\Psi(t)\}=0.
\end{equation}
\end{lemma}

\begin{proof}
By condition (1), (2) of ($a$) in Definition \ref{Definition 4}, for any $g\in\Psi(t)$, $f\in L^{\infty}(\mathbb{D})$, $h\in L_a^2(\mathbb{D})$,
$$\|T_{fg}h-T_fh\|^2\leq \|fgh-fh\|^2\leq \|f\|_{\infty}\int_{\mathbb{D}\setminus \overline{\mathbb{D}_t}}|h(z)|^2dv(z)\to 0\text{ as }t\to 1.$$
\end{proof}

The following key lemma from \cite{MuhlyandXia2000} would be used frequently in this paper.

\begin{lemma}\label{Lemma 6}(\cite{MuhlyandXia2000} Lemma 2.1)

Let $\{B_i\}$ be a sequence of compact operators on a Hilbert space $\mathcal{H}$ satisfying the following coditions:

(a) Both sequences $\{B_i\}$ and $\{B_i^*\}$ converge to $0$ in the strong operator topology.

(b) The limit $\lim_{i\to \infty}\|B_i\|$ exists.

Then there exist natural  numbers $i(1)<i(2)<\cdots$ such that the sum
$$\sum_{m=1}^{\infty}B_{i(m)}=\lim_{N\to\infty}\sum_{m=1}^NB_{i(m)}$$
exists in the strong operator topology and we have
$$\|\sum_{m=1}^{\infty}B_{i(m)}\|_{\mathcal{Q}}=\lim_{i\to\infty}\|B_i\|.$$
\end{lemma}

\begin{definition}\label{Definition 5}

For $0<t<1$ and $\delta>0$, the symbol $\Phi(t;\delta)$ denotes the collection of continuous functions $f$ on $\mathbb{D}$ satisfying:

(1) $0\leq f(z)\leq 1$ on $\mathbb{D}$;

(2) $f(z)=0$ whenever $z\in \overline{\mathbb{D}_t}$;

(3) $\mathrm{diff}(f)\leq\delta$.
\end{definition}

\begin{lemma}\label{Lemma 8} Let $\{h_k\}_{k\in\mathbb{N}_+}$ be a sequence of continuous functions on $\mathbb{D}$ and denote $U_k=\{z\in\mathbb{D}:h_k(z)\neq 0\}$ for any $k\in\mathbb{N}_+$. Suppose that this sequence has the property that:
\begin{equation}\label{8.1}
\text{there is an }\, a>1\text{ such that }\inf\{d(z,w):z\in U_k, w\in U_j\text{ for any }j\neq k\}\geq a.
\end{equation}
Then the function $h=\sum_{k=1}^{\infty}h_k$ satisfies $\mathrm{diff}(h)\leq\sup_{k\in\mathbb{N}_+}\{\mathrm{diff}(h_k)\}$.
\end{lemma}

\begin{proof}
For any $z,w\in\mathbb{D}$ and $d(z,w)\leq 1$, by (\ref{8.1}), we only need to notice that there is at most one $k\in\mathbb{N}_+$ such that $h_k(z)-h_k(w)\neq 0$.
\end{proof}

\begin{lemma}\label{Lemma 7}
Let $\{f_k\}_{k\in\mathbb{N}_+}\subset C(\mathbb{D})$ satisfying the following conditions:
\begin{itemize}
  \item (1) There exists $0<C<\infty$ such that $\|f_k\|_{\infty}\leq C$ for any $k\in\mathbb{N}_+$;
  \item (2) For every $k\in\mathbb{N}_+$, there exist $0<a_k<b_k<1$ such that $f_k=0$ on $\overline{\mathbb{D}_{a_k}}\cup (\mathbb{D}\setminus \mathbb{D}_{b_k})$;
  \item (3) $\lim_{k\to\infty}a_k=1$;
  \item (4) $\lim_{k\to\infty} \mathrm{diff}(f_k)=0$.
\end{itemize}
Then there exists an infinite subset $I\subset \mathbb{N}_+$ such that $f_J\doteq\sum_{k\in J}f_k\in VO_{bdd}$ for any $J\subset I$.
\end{lemma}

\begin{proof}
For any $z,w\in D$, we have
\begin{equation}\label{7.1}
\beta(z,w)\geq |\beta(w,0)-\beta(z,0)|=\frac{1}{2}|\log\frac{(1+|w|)(1-|z|)}{(1-|w|)(1+|z|)}|.
\end{equation}
By $\lim_{k\to\infty}a_k=1$, we can inductively choose $I=\{k(1),\cdots,k(j),\cdots\}$ such that $k(1)<\cdots<k(j)<\cdots$ and $a_{k(j)}<b_{k(j)}<a_{k(j+1)}<b_{k(j+1)}$ for any $j\in\mathbb{N}_+$ such that
\begin{equation}\label{7.2}
\frac{1}{2}\log\frac{(1+a_{k(j+1)})(1-b_{k(j)})}{(1-a_{k(j+1)})(1+b_{k(j)})}\geq 2.
\end{equation}
For each $k\in\mathbb{N}_+$, let $R_k=\overline{\mathbb{D}_{b_k}}\setminus \mathbb{D}_{a_k}$.
Then by condition (2), we can get that $f_k|_{\mathbb{D}\setminus R_k}=0$.

It follows from (\ref{7.1}), (\ref{7.2}) that if $z\in R_{k(j)}$, $w\in R_{k(j')}$ with $j\neq j'$, then
\begin{equation}\label{7.3}
d(z,w)=\sum_{i=1}^n\beta(z_i,w_i)\geq 2.
\end{equation}
Combining with condition (1) and (2), we have that $f_J$ is continuous on $\mathbb{D}$ and $\|f_J\|_{\infty}\leq C$ for any $J\subset I$.

Let $j_0\in \mathbb{N}_+$ and let $z,w\in \mathbb{D}$ satisfying $d(z,w)\le 1$ and $z\in \mathbb{D}\setminus \mathbb{D}_{a_{k(j_0)}}$. By (\ref{7.3}), there exists at most one $j\in\mathbb{N}_+$, such that $|f_{k(j)}(z)-f_{k(j)}(w)|\neq 0$. It follows from (\ref{7.1}) and (\ref{7.2}) that $z,w\in R_{k(j)}$ and $j\geq j_0$. Then by Lemma \ref{Lemma 8}, we have that
\begin{equation}\label{youjiebiancha}
|f_J(z)-f_J(w)|\leq \sup\{\mathrm{diff}(f_{k(j)}):j\geq j_0\}.
\end{equation}
By condition (3), we have that $d(0,z)\rightarrow \infty$ as $j_0 \rightarrow \infty$. Combining with condition (4) and (\ref{youjiebiancha}), we can obtain that $f_J\subset VO_{bdd}$.

\end{proof}

\begin{proposition}\label{Proposition L}

Let $X$ be an operator in $EssCom\{T_g:g\in VO_{bdd}\}$.
Then for any $\varepsilon>0$, there exists $\delta=\delta(X,\varepsilon)>0$ such that
$$\lim_{t\to 1}\sup\{\|[X,T_f]\|:f\in \Phi(t;\delta)\}\leq \varepsilon.$$
\end{proposition}

\begin{proof}
If not, then for some fixed $\varepsilon>0$, there exists $f_k\in \Phi(1-\frac{1}{k};\frac{1}{k})$ for any $k\in\mathbb{N}_+$ such that $\|[X, T_{f_k}]\|>\varepsilon$. Then for every $k\in\mathbb{N}_+$, there exist $h_k,\psi_k\in L_a^2(\mathbb{D})$ with $\|h_k\|=\|\psi_k\|=1$ such that
\[
|\langle [X,T_{f_k}]h_k,\psi_k\rangle|>\varepsilon.
\]
Applying Lemma \ref{Lemma 5}, we see that for every $k\in\mathbb{N}_+$, there is a $1 -\frac{1}{k} < t_k < 1$ such
that
\[
\langle [X, T_{f_kg}]h_k, \psi_k \rangle| > \varepsilon \text{ for every }g\in\Psi(t_k).
\]
By Lemma \ref{Lemma 4}, there exists $g_k\in \Psi(t_k;\frac{1}{k})$ for any $k\in\mathbb{N}_+$. Let $q_k=f_kg_k$ for any $k\in\mathbb{N}_+$ and we have $$|\langle[X,T_{q_k}]h_k,\psi_k\rangle|>\varepsilon.$$
Since $h_k,\psi_k$ are unit vectors, we have
\begin{equation}\label{L.1}
\|[X,T_{q_k}]\|>\varepsilon, \,\text{for every}\, k\in\mathbb{N}_+.
\end{equation}
Since $q_k=f_kg_k$ with $f_k\in \Phi(1-\frac{1}{k},\frac{1}{k})$, $g_k\in \Psi(t_k;\frac{1}{k})$ for any $k\in \mathbb{N}_+$, then the function $q_k$ satisfies $0\leq q_k\leq 1$ on $\mathbb{D}$ and
\[
\text{diff}(q_k)\leq\text{diff}(f_k)+\text{diff}(g_k)\leq\frac{2}{k}.
\]
and there exists $t_k< t_k'<1$ such that
 $g_k|_{\mathbb{D}\setminus \mathbb{D}_{t_k'}}=0$ and then \begin{equation}\label{qkk}
q_k|_{\mathbb{D}_{1-\frac{1}{k}}\cup(\mathbb{D}\setminus \mathbb{D}_{t_k'})}=0.
\end{equation}
By Lemma \ref{Lemma 7}, there exists an infinite set $I\subset \mathbb{N}_+$, such that
\begin{equation}\label{q_kinVObdd}
q_J=\sum_{k\in J}q_k\in VO_{bdd}\text{ for any } J\subset I.
\end{equation}
Since $\|f_k\|_{\infty}\leq 1$ and $\|g_k\|_{\infty}\leq 1$, then $\|q_k\|_{\infty}=\|f_kg_k\|_{\infty}\leq 1$ and hence
$$\|[X,T_{q_k}]\|\leq2\|X\|\text{ for any }k\in\mathbb{N}_+,$$
which means that $\{\|[X,T_{q_k}]\|\}_k$ is bounded. Therefore, we can choose an increasing sequence $\{k_i\}_{i=1}^{\infty}\subset I$ such that $\lim_{i\to\infty}\|[X,T_{q_{k_i}}]\|=d$, which, by (\ref{L.1}), is equal or greater than $\varepsilon$. By (\ref{qkk}), we have $\mathrm{supp}(q_{k_i})\subseteq \mathbb{D}_{t_k'}$. Then $B_i\doteq[X,T_{q_{k_i}}]$ is compact for any $i\in\mathbb{N}_+$ and hence we have $B_i$ and $B_i^*$ converges to 0 in strong operator topology and $d=\lim_{i\to\infty}\|B_i\|$.
By Lemma \ref{Lemma 6}, there exist increasing subsequence $\{i(m)\}_{m=1}^{\infty}$ such that $B=\lim_{N\to\infty}\sum_{m=1}^NB_{i(m)}$ exists in the strong operator topology with $\|B\|_{\mathcal{Q}}=d\geq \varepsilon>0$.

Define $E=\{k_{i(m)}\}_{m=1}^{\infty}$. Then $E\subseteq \{k_1,\dots,k_i,\dots\}\subseteq I$, we have $q_E\in VO_{bdd}$ by (\ref{q_kinVObdd}). Since $q_E$ is bounded and every $q_{k_{i(m)}}$ is non-negative, by the dominated convergence theorem, in strong operator topology we have
\[
T_{q_E}=\lim_{N\to\infty}T_{q_{k_{i(1)}}+\cdots+q_{k_{i(N)}}}=\lim_{N\to\infty}\sum_{m=1}^NT_{q_{k_{i(m)}}}
\]
Thus
\[
B=\lim_{N\to\infty}\sum_{m=1}^NB_{i(m)}=\lim_{N\to\infty} [X,\sum_{m=1}^N T_{q_{k_{i(m)}}}]=[X,T_{q_E}]\text{ with }q_E\in VO_{bdd}.
\]
Since $X\in EssCom\{T_g:g\in VO_{bdd}\}$, then $B$ is compact which contradicts with $\|B\|_{\mathcal{Q}}=d>0$.
\end{proof}

%

For any $z\in D$ and non-empty subset $W\subset D$, we denote
\[
\beta(z,W)=\inf\{\beta(z,\xi):\xi\in W\}.
\]
For all $z, w\in D$, we have
\[
|\beta(z,W) - \beta(w,W)|\leq \beta (z,w).
\]
For every $m \geq 6$, define the function $\tilde{f}_m$ on $\mathbb{R}_+\cup\{0\}$ by
\[
\tilde{f}_m(x) =
\begin{cases}
1-(\log m)^{-1}x  &\quad\text{ if }0 \leq x \leq \log m\\
0 &\quad\text{ if }\log m<x
\end{cases}.
\]

Obviously, this function satisfies the Lipschitz condition $|\tilde{f}_m(x)-\tilde{f}_m(y)| \leq (\log m)^{-1}|x-y|$ for all $x, y \in [0,\infty)$. Given any triple of $m \geq 6$, $j\in\mathbb{N}$ and $u \in E_{m,j}$, we now define
\[
f_{m,j,u}(z)=\tilde{f}_m(\beta(z,A_{m,j,u})) \text{ for any }z\in D.
\]

Let $S\doteq \{(m,j,u):j\in\mathbb{N},m\geq 6,u\in E_{m,j}\}$. The main parts of following Lemma
 in \cite{J. Xia 2018} for the set $S$ without $(m,0,0)$. We prove only for (5) and $(m,0,0)$.

\begin{lemma}\label{Lemma 9}\cite{J. Xia 2018}
For any $w\in S$, the function $f_w$ has the following five properties:
\begin{itemize}
  \item (1) The inequality $0 \leq f_w \leq 1$ holds on $D$ for any $w\in S$;
  \item (2) $f_w= 1$ on the set $A_w$ for any $w\in S$;
  \item (3) $f_w$ is continuous on $D$ for any $w\in S$;
  \item (4) $\mathrm{supp}(f_w)\subset B_w$ for any $w\in S$;
  \item (5) $\mathrm{diff}(f_w) \leq (\log m)^{-1}$ for any $w\in S$.
\end{itemize}
\end{lemma}

\begin{proof}
For any $w\in S\setminus \{(m,0,0)\}$, (1),(2),(3),(5) follows directly from the definition of $f_w$. For (4), note that if $z\in D\setminus B_w$, then (c),(d) of Lemma \ref{Lemma 2} give us $\beta(z,A_w)\geq 2\log m$ and hence $\tilde{f}_m(\beta(z,A_w))=0$ whenever $w\in S\setminus \{(m,0,0)\}$.

For any $w=(m,0,0)$, (1),(2),(3),(5) follows directly from the definition of $f_w$. For (4), note that if $z\in D\setminus B_{m,0,0}$, then $z\in\{w\in D:|z|\geq \rho_{5m}\}=\bigcup_{j\geq 3} \bigcup_{u\in E_{m,j}}A_{m,j,u}$. Then  by (d) of Lemma \ref{Lemma 2}, $\beta(z,A_{m,0,0})\geq 2\log m$ and hence $\tilde{f}_m(\beta(z,A_{m,0,0}))=0$, which implies supp$(f_{m,0,0})\subset B_{m,0,0}$.
\end{proof}

For any $\hat{w}\in S^n$ and any $z=(z_1,\cdots,z_n)\in\mathbb{D}$, define
\begin{equation}\label{hanshu}\hat{f}_{\hat{w}}(z)=\prod_{i=1}^nf_{w_i}(z_i),\end{equation}
where $w_i\in S$ for any $1\leq i\leq n$. As the product of functions $f_{w_i}$, Lemma \ref{Lemma 9} implies:

\begin{lemma}\label{Lemma 9'}
For every $\hat{w}\in S^n\backslash\{(m,0,0),\dots,(m,0,0)\}$, the function $\hat{f}_{\hat{w}}$ has the following five properties:
\begin{itemize}
  \item (a) The inequality $0 \leq \hat{f}_{\hat{w}} \leq 1$ holds on $\mathbb{D}$.
  \item (b) $\hat{f}_{\hat{w}}= 1$ on the set $\hat{A}_{\hat{w}}$.
  \item (c) $\hat{f}_{\hat{w}}$ is continuous on $\mathbb{D}$.
  \item (d) $\mathrm{supp}(\hat{f}_{\hat{w}})\subset \hat{B}_{\hat{w}}$.
  \item (e) $\mathrm{diff}(\hat{f}_{\hat{w}}) \leq n\cdot\max_{1\leq i\leq n}(\mathrm{diff}(f_{w_i}))\leq n\cdot(\log m)^{-1}$.
\end{itemize}
%
\end{lemma}

\begin{definition}\label{Definition 6}
Let $m\geq 6$. For any subset $\hat{I}\subset \hat{I}_m$ which appear in Definition \ref{Definition 1}, denote
\[
\hat{f}_{\hat{I}}=\sum_{\hat{w}\in\hat{I}}\hat{f}_{\hat{w}}, ~~~\hat{F}_{\hat{I}}=\sum_{\hat{w}\in\hat{I}}\hat{f}_{\hat{w}}^2.
\]

\end{definition}

\begin{lemma}\label{Lemma 10}
 Let $m\geq 6$, $\hat{k}\in \{0, 1, 2, 3, 4, 5\}^n$ and $\hat{\nu}\in\{1,\cdots, \widetilde{N}_0\}^n$, where $\hat{I}^{(\hat{\nu},\hat{k})}_m$, $\hat{k}$, $\hat{\nu}$ and $\widetilde{N}_0$ appear in Definition \ref{Definition 1}. Then for every subset
$\hat{I}$ of $\hat{I}^{(\hat{\nu},\hat{k})}_m$, we have $\hat{f}_{\hat{I}}\in \Phi(\rho_m;n(\log m)^{-1})$.
\end{lemma}
\begin{proof}
For any $\hat{w}\in \hat{I}$, Lemma \ref{Lemma 9'} implies that
\begin{equation}\label{10.0}
\text{supp}(\hat{f}_{\hat{w}})\subset \hat{B}_{\hat{w}}\text{ and }\mathrm{diff}(\hat{f}_{\hat{w}})\leq n(\log m)^{-1}.
\end{equation}

For $\hat{w}\neq \hat{w}'\in \hat{I}^{(\hat{\nu},\hat{k})}_m$, there exists $i_0\in \{1,\cdots, n\}$ such that
$$w_{i_0}=(m,6p+k_{i_0},u)\neq w_{i_0}'=(m,6q+k_{i_0},v)\in I_m^{(\nu_{i_0},k_{i_0})},$$
in which $p,q\in\mathbb{N}$ and $u\in E_{m,6p+k_{i_0}}^{(\nu_{i_0})}$ and $v\in E_{m,6q+k_{i_0}}^{(\nu_{i_0})}$.

If $k_{i_0}\neq 0$, then we have $p\neq q$ or $u\neq v$. In any case, by Lemma \ref{Lemma 2} (a) and (b), we have
$$d(\hat{B}_{\hat{w}},\hat{B}_{\hat{w}'})\geq \beta(B_{w_{i_0}},B_{w_{i_0}'})\geq 2.$$

If $k_{i_0}= 0$ and neither $w_{i_0}$ nor $w_{i_0}'$ equals to $(m,0,0)$, then also by Lemma \ref{Lemma 2} (a) and (b),
$$d(\hat{B}_{\hat{w}},\hat{B}_{\hat{w}'})\geq \beta(B_{w_{i_0}},B_{w_{i_0}'})\geq 2.$$

Otherwise, $k_{i_0}= 0$ and either $w_{i_0}$ or $w_{i_0}'$ equals to $(m,0,0)$. Without lose of generality, we can assume $w_{i_0}= (m,0,0),w_{i_0}'= (m,6j_{i_0},u_{j_{i_0}})$ for some $j_{i_0}\in \mathbb{N}_+$ and some $u_{j_{i_0}}\in E_{m,6j_{i_0}}^{(\nu_{i_0})}$. In this case, for $m\geq 6$, we have
\begin{multline*}
d(\hat{B}_{\hat{w}},\hat{B}_{\hat{w}'})\geq \beta(B_{w_{i_0}},B_{w_{i_0}'})\geq \beta(\{z\in D:|z|<\rho_{5m}\},\{z\in D:|z|>\rho_{6m}\}) \\
=\frac{1}{2}\log\frac{(1+\rho_{6m})(1-\rho_{5m})}{(1-\rho_{6m})(1+\rho_{5m})}\geq \frac{2m-1}{4}\log 4\geq 2,
\end{multline*}
%
%
%
then for  any $\hat{w}\neq \hat{w}'\in \hat{I}^{(\hat{\nu},\hat{k})}_m$,
\begin{equation}\label{10.1}
d(\hat{B}_{\hat{w}},\hat{B}_{\hat{w}'})\ge 2.
\end{equation}
Therefore, by Lemma \ref{Lemma 8}, Lemma \ref{Lemma 9'} (d) and (\ref{10.0}), we have
\begin{equation}\label{10.2}
\hat{f}_{\hat{I}}\in C(\mathbb{D}),~~0\leq \hat{f}_{\hat{I}}\leq 1\text{  and diff }(\hat{f}_{\hat{I}})\leq\sup_{\hat{w}\in\hat{I}}(\text{ diff }(\hat{f}_{\hat{w}}))\leq n(\log m)^{-1}.
\end{equation}

Finally, for any $z\in \mathbb{D}_{\rho_m}$ and any $\hat{w}\in\hat{I}$, by the construction of $\hat{I}_m^{(\hat{\nu},\hat{k})}$, there exists $i_0\in\{1,\cdots,n\}$, such that $w_{i_0}=(m,6j_{i_0}+k_{i_0},u_{i_0})\neq (m,0,0)$. So by Lemma \ref{Lemma 9} (4) and the fact $B_{w_{i_0}}\cap \{z\in D:|z|<\rho_m\}=\emptyset$, we have $f_{w_{i_0}}(z_{i_0})=f_{m,6j_{i_0}+k_{i_0},u_{i_0}}(z_{i_0})=0$, for any $z_{i_0}\in D$ with $|z_{i_0}|<\rho_m$. Therefore, for any $z\in \mathbb{D}_{\rho_m}$ and any $\hat{w}\in\hat{I}$,
$$\hat{f}_{\hat{w}}(z)=\prod_{i=1}^nf_{w_i}(z_i)=0$$
and hence
$$\hat{f}_{\hat{I}}(z)=\sum_{\hat{w}\in\hat{I}}\hat{f}_{\hat{w}}(z)=0, \text{ for any }z\in \mathbb{D}_{\rho_m},$$
which leads to the conclusion combined with (\ref{10.2}).
\end{proof}

\section{Membership of $LOC(X)$ }
\

In this section, we would give some criteria for the membership of $LOC(X)$. We begin with

\begin{lemma}\label{Lemma 11}
Let $m\geq 6$, $\hat{k}\in\{0,\cdots,5\}^n$, $\hat{\nu}\in\{1,\cdots,\widetilde{N}_0\}^n$, where $\hat{k}$, $\hat{\nu}$, $\widetilde{N}_0$, $\hat{I}_m^{(\hat{\nu},\hat{k})}$ appear in Definition \ref{Definition 1}. Then for every subset $\hat{I}$ of $\hat{I}_m^{(\hat{\nu},\hat{k})}$, for every bounded operator $X$ on $L_a^2(\mathbb{D})$, we have $\sum_{\hat{w}\in\hat{I}}T_{\hat{f}_{\hat{w}}}XT_{\hat{f}_{\hat{w}}}\in LOC(X)$, where $\hat{f}_{\hat{w}}$ appears in (\ref{hanshu}).
\end{lemma}
\begin{proof}
For any $\hat{I}\subset \hat{I}_m^{(\hat{\nu},\hat{k})}$, let $\hat{\Gamma}=\{z_{\hat{w}}=(z_{w_1},\cdots,z_{w_n}):\hat{w}=(w_1,\cdots,w_n)\in \hat{I}\}$, where $z_{w_i}$ appears in (\ref{z}). So $z_{\hat{w}}\in B_{\hat{w}}$ for any $\hat{w}\in \hat{\Gamma}$ and hence $\hat{\Gamma}$ is separated by (\ref{10.1}). Rewrite $\sum_{\hat{w}\in\hat{I}}T_{\hat{f}_{\hat{w}}}XT_{\hat{f}_{\hat{w}}}$ to $\sum_{z_{\hat{w}}\in\hat{\Gamma}}T_{\tilde{f}_{z_{\hat{w}}}}XT_{\tilde{f}_{z_{\hat{w}}}}$, where $\tilde{f}_{z_{\hat{w}}}=\hat{f}_{\hat{w}}$.

%
By Lemma \ref{Lemma 9'}, we have
\begin{multline*}
|\tilde{f}_{z_{\hat{w}}}(\xi)-\tilde{f}_{z_{\hat{w}}}(\xi')|\leq |\hat{f}_{\hat{w}}(\xi)-\hat{f}_{\hat{w}}(\xi')|\leq \sum_{i=1}^n |f_{z_{w_i}}(\xi_i)-f_{z_{w_i}}(\xi_i')|  \\
\le \sum_{i=1}^n \frac{1}{\log m}|\beta(\xi_i,A_{w_i})-\beta(\xi_i',A_{w_i})|\le (\log m)^{-1}d(\xi,\xi').
\end{multline*}
and $\mathrm{supp}(\tilde{f}_{z_{\hat{w}}})\subset \hat{B}_{\hat{w}}$, where $\hat{B}_{\hat{w}}$ appears in (\ref{liangyong}). Combining this with Lemma \ref{Lemma 3'}, we have $\tilde{f}_{z_{\hat{w}}}=0$ on $\mathbb{D}\setminus \widetilde{D}(z_{\hat{w}}, L_m)$.
\end{proof}

For any $f\in L^{\infty}(\mathbb{D})$, define Hankel operator
\[
H_fh=(1-P)(fh),~\text{ for any } h\in L_a^2(\mathbb{D}).
\]

\begin{lemma}\label{Lemma 12}
There exists a $0<C<\infty$, such that $\|H_f\|< C\mathrm{diff}(f)$, for every bounded continuous function $f$ on $\mathbb{D}$.
\end{lemma}
 \begin{proof}
We have known that there is a constant $C_1$ such that $\|H_f\| \le C_1 \|f\|_{BMO}$ for every bounded continuous function $f$ on $\mathbb{D}$ by Theorem 22 in \cite{CoburnZhu1990}. Then it suffices to show that $\|f\|_{BMO}\le C_2\mathrm{diff}(f)$ for some $C_2>0$.

Similar with the proof of Lemma 4.8 in \cite{J. Xia 2018},
for any $j\in\mathbb{N}_+$, $t\in \mathbb{R}$ with $1-2^{-j} \leq t \leq 1-2^{-j-1}$ and $\xi_1\in S^1$, we have
\begin{equation}\label{yinyong}
\beta((1-2^{-j})\xi_1, t\xi_1)=\frac{1}{2} \log\frac{(1 + t)2^{-j}}{(1-t)(2-2^{-j} )}\leq \frac{1}{2} \log 4 < 1.
\end{equation}
Further, by (\ref{yinyong}) and an obvious telescoping sum, for any continuous function $g$ on $D$, we see that
\[
|g(u)-g(0)| \leq (j+1)\sup\{|g(u)-g(v)|:\beta(u,v)<1, u,v\in D \} \text{ for every }u \in Q_j, \text{ for any }j \in \mathbb{N}_+,
\]
where $Q_j = \{u \in D : 1- 2^{-j} \leq |u| < 1-2^{-j-1}\}.$ Write the sum $C=\left(\sum_{j=0}^{\infty}(j+1)^2 A(Q_j)\right)^{\frac12}$, which is finite by the construction of $Q_j$.
For any bounded continuous function $f$ on $\mathbb{D}$, we have
\begin{eqnarray*}
\|f-f(0)\|^2&=&\int_{\mathbb{D}} |f(w)-f(0)|^2dv(w)\\
&\le& \int_{\mathbb{D}}(|f(w)-f(w^{(1)})|+|f(w^{(1)})-f(w^{(2)})|+\dots+|f(w^{(n-1)})-f(0)|)^2dv(w)\\
&\le& \int_{\mathbb{D}} n(|f(w)-f(w^{(1)})|^2+|f(w^{(1)})-f(w^{(2)})|^2+\dots+|f(w^{(n-1)})-f(0)|^2)dv(w)\\
&\le&\int_{D^{n-1}}(\int_D n|f(w)-f(w^{(1)})|^2dA_1(w_1))dA_2(w_2)\dots dA_n(w_n)+\dots \\
&\qquad & +\int_{D^{n-1}}(\int_D n|f(w^{(n-1)})-f(0)|^2dA_n(w_n))dA_1(w_1)\dots dA_{n-1}(w_{n-1})\\
&\le & (nC\mathrm{diff}(f))^2,
\end{eqnarray*}
where $w=(w_1,w_2,w_3,\dots,w_n),\, w^{(1)}=(0,w_2,w_3,\dots,w_n),\, w^{(2)}=(0,0,w_3,\dots,w_n),\dots, w^{(n-1)}=(0,0,\dots,0,w_n).$
For each $z\in \mathbb{D}$, denote $f_z=f\circ \varphi_z$. By the M\"{o}bius invariance of $\beta$, we have
$$\|(f-\langle fk_z,k_z\rangle)k_z\| \le \|f-f(z)\|=\|f_z-f_z(0)\|\le nC\mathrm{diff}(f_z)=nC\mathrm{diff}(f).$$
Since polydisk $\mathbb{D}$ is a bounded symmetric
domains in $\mathbb{C}^n$, then we have
$$\|f\|_{BMO}=\sup_{z\in \mathbb{D}}\|(f-\langle fk_z,k_z\rangle)k_z\|$$
 in \cite{CoburnZhu1990} and hence $\|f\|_{BMO}\le nC\mathrm{diff}(f)$.
\end{proof}
Completely analogous to the properties of Toeplitz operators on open unit ball $\mathbb{B}$, we have the following Lemma, the proof is essentially the same as Lemma 5.1 in \cite{J. Xia 2018}.
\begin{lemma}\label{Lemma 13}Let $\{f_1,\cdots, f_l\}$ be a finite set of functions in $L^{\infty}(\mathbb{D})$ with the property
that $f_jf_k =0$ for all $j\neq k$ in $\{1, \cdots, l\}$. Let $A$ be any bounded operator on the Bergman
space $L^2_a(\mathbb{D})$. Then there exist complex vectors $\{\gamma_1,\cdots, \gamma_l\}$ with $|\gamma_k| = 1$ for every
$k\in\{1, \cdots, l\}$ and a subset $E$ of $\{1, \cdots, l\}$ such that if we define
\[
F=\sum_{k\in E}f_k,~~ G =\sum_{k\in\{1,\cdots,l\}\setminus E}f_k, ~~F' =\sum_{k\in E}\gamma_kf_k \text{ and }G' =\sum_{k\in\{1,\cdots,l\}\setminus E}
\gamma_kf_k
\]
then
\[
\|\sum_{j\neq k}T_{f_j}AT_{f_k}\|\leq 4(\|T_{F'}AT_{G}\|+\|T_{G'}AT_{F}\|).
\]
\end{lemma}

\begin{proposition}\label{Proposition M}
Let $X\in EssCom(\{T_g:g\in VO_{bdd}\})$, then $X\in \overline{Span\{LOC(X)+\mathcal{K}\}}$, where the closure is taken in operator norm.
\end{proposition}
\begin{proof}

For any $\varepsilon>0$, by Proposition \ref{Proposition L}, then there exists $ \delta>0$ and $0<t^*<1$ such that for any $f\in \Phi(t^*;\delta)$,

\begin{equation}\label{m0}
  \|[X,T_f]\|\leq 2\varepsilon.
\end{equation}

Choose an integer $m\geq 6$ satisfying

\begin{equation}\label{m3}
 n(\log m)^{-1}((6N_0)^n+2)\leq \min\{\varepsilon,\delta\}\text{ and }\rho_m\geq t^*.
\end{equation}








Let $\hat{f}_0\in \Psi(\rho_{3m},n(\log m )^{-1})$. Since by (\ref{10.1}), $B_{\hat{w}}\cap B_{\hat{w}'}=\emptyset$, $\text{for any } \hat{w}\neq \hat{w}'\in \hat{I}_m^{(\hat{\nu},\hat{k})}$, then Lemma \ref{Lemma 9'} implies $0\leq \hat{F}_{\hat{I}_m^{(\hat{\nu},\hat{k})}}\leq 1$ on $\mathbb{D}$, where $\hat{F}_{\hat{I}_m^{(\hat{\nu},\hat{k})}}$ appears in Definition \ref{Definition 6}, and hence $0\leq \hat{F}_{\hat{I}_m}\leq (6N_0)^n$. Again by Lemma \ref{Lemma 9'} (b),
$$\hat{F}_{\hat{I}_m}(z)\geq 1 \,\text{ for any }\, x\in\mathop\bigcup \limits_{\hat{\omega} \in \hat{I}_m}\hat{A}_{\hat{\omega}}=(\hat{A}_{m,\hat{0},\hat{0}})^c,$$
where $\hat{A}_{m,\hat{0},\hat{0}}$ appears in (\ref{yuji}),
and hence
\begin{equation}\label{uplowbound}
1\leq \hat{f}_{0}^2(z)+\hat{F}_{\hat{I}_m}(z)\leq (6N_0)^n+1\,~\text{on }\, \mathbb{D}.
\end{equation}
Define $h_m(z)=\hat{f}_{0}^2(z)+\hat{F}_{\hat{I}_m}(z)$. Since \begin{equation}\label{bianchayiju1}\mathrm{diff}(\hat{f}_{0})\leq n (\log m)^{-1}\, \text{and}\, \mathrm{diff}(\hat{f}_{\hat{I}})\leq n (\log m)^{-1}\end{equation} for any $\hat{I}\subset \hat{I}_m^{(\hat{\nu},\hat{k})}$, $\hat{k}\in \{0, 1, 2, 3, 4, 5\}^n$ and   $\hat{\nu}\in\{1,\cdots, \widetilde{N}_0\}^n$, by Lemma \ref{Lemma 10}, we have

$$\text{diff}(h_m)=
\text{diff}(\hat{f}_0^2+\sum_{\hat{\nu}}\sum_{\hat{k}}\sum_{\hat{w}\in\hat{I}_m^{(\hat{\nu},\hat{k})}}\hat{f}^2_{\hat{w}})\leq 2n (\log m)^{-1}((6N_0)^n+1).$$
Combining with (\ref{uplowbound}), we have
\begin{equation}\label{bianchayiju2}
 \text{diff}(\frac{1}{\sqrt{h_m}})\leq n (\log m)^{-1}((6N_0)^n+1).
\end{equation}
For any $\hat{w}\in \hat{I}_m$, any $\hat{I}\subset \hat{I}_m^{(\hat{\nu},\hat{k})}$ for any $\hat{k}\in \{0, 1, 2, 3, 4, 5\}^n$ and $\hat{\nu}\in\{1,\cdots, \widetilde{N}_0\}^n$, we define

\[
\tilde{f}_{0}=\frac{\hat{f}_{0}}{\sqrt{h_m}},~~\tilde{f}_{\hat{w}}=\frac{\hat{f}_{\hat{w}}}{\sqrt{h_m}}\text{ and }\tilde{f}_{\hat{I}}=\sum_{\hat{w}\in\hat{I}}\tilde{f}_{\hat{w}}.
\]
It is easy to see $\tilde{f}_{0}^2+\sum_{\hat{\nu}}\sum_{\hat{k}}\tilde{f}_{\hat{I}_m^{(\hat{\nu},\hat{k})}}^2\equiv 1$. According to (\ref{bianchayiju1}) and (\ref{bianchayiju2}), we have
\begin{equation}\label{bianchayiju3}
\text{diff}(\tilde{f}_{\hat{w}})\leq n (\log m)^{-1}((6N_0)^n+2)\text{ and }\text{diff}(\tilde{f}_{\hat{I}})\leq n (\log m)^{-1}((6N_0)^n+2),
\end{equation}
for any $\hat{w}\in \hat{I}_m$, any $\hat{I}\subset \hat{I}_m^{(\hat{\nu},\hat{k})}$ for any $\hat{k}\in \{0, 1, 2, 3, 4, 5\}^n$ and $\hat{\nu}\in\{1,\cdots, \widetilde{N}_0\}^n$.
Combining with (\ref{m0}), we have
\begin{equation}\label{m-1}
\|[X,T_{\tilde{f}_{\hat{I}}}]\|\leq 2\varepsilon,
\end{equation}
for any $\hat{I}\subset \hat{I}_m^{(\hat{\nu},\hat{k})}$ for any $\hat{k}\in \{0, 1, 2, 3, 4, 5\}^n$ and $\hat{\nu}\in\{1,\cdots, \widetilde{N}_0\}^n$.

Then for the fixed $X\in EssCom(\{T_g:g\in VO_{bdd}\})$, we have
\begin{equation}\label{X}
X=XT_1=XT_{\tilde{f}_{0}^2+\sum_{\hat{\nu}}\sum_{\hat{k}}\tilde{f}_{\hat{I}_m^{(\hat{\nu},\hat{k})}}^2}=XT_{\tilde{f}_{0}^2}+\sum_{\hat{\nu}}\sum_{\hat{k}}X_{(\hat{\nu},\hat{k})},
\end{equation}
where $X_{(\hat{\nu},\hat{k})}=XT_{\tilde{f}_{\hat{I}_m^{(\hat{\nu},\hat{k})}}^2}$. Since supp$(\tilde{f}_{0})$ is compact, then $XT_{\tilde{f}_{0}^2}\in\mathcal{K}\subset \mathcal{T}(\mathbb{D})$.

For $X_{(\hat{\nu},\hat{k})}$, by (\ref{10.1}) and (d) of Lemma \ref{Lemma 9'}, we have $\mathrm{supp}(\tilde{f}_{\hat{w}})\cap \mathrm{supp}(\tilde{f}_{\hat{w'}})=\emptyset$ for any distinct elements $w,w'$ in $\hat{I}_m^{(\hat{\nu},\hat{k})}$. Then

\[
T_{\tilde{f}^2_{\hat{I}_m^{(\hat{\nu},\hat{k})}}}=T^2_{\tilde{f}_{\hat{I}_m^{(\hat{\nu},\hat{k})}}}+H^*_{\tilde{f}_{\hat{I}_m^{(\hat{\nu},\hat{k})}}}H_{\tilde{f}_{\hat{I}_m^{(\hat{\nu},\hat{k})}}}.
\]
Write
$$X_{(\hat{\nu},\hat{k})}=X_{(\hat{\nu},\hat{k})}^{(1)}+Z_{(\hat{\nu},\hat{k})}^{(1)},$$
where
$$X_{(\hat{\nu},\hat{k})}^{(1)}=XT^2_{\tilde{f}_{\hat{I}_m^{(\hat{\nu},\hat{k})}}}~~\text{  and   }~~Z_{(\hat{\nu},\hat{k})}^{(1)}=XH^*_{\tilde{f}_{\hat{I}_m^{(\hat{\nu},\hat{k})}}}H_{\tilde{f}_{\hat{I}_m^{(\hat{\nu},\hat{k})}}}.$$
By Lemma \ref{Lemma 12}, we can find that there exists $C>0$ such that
\begin{equation}\label{Z1}
\|Z_{(\hat{\nu},\hat{k})}^{(1)}\|\leq C\|X\|\varepsilon.
\end{equation}
Write
$$X_{(\hat{\nu},\hat{k})}^{(1)}=X_{(\hat{\nu},\hat{k})}^{(2)}+Z_{(\hat{\nu},\hat{k})}^{(2)},$$
where
$$Z_{(\hat{\nu},\hat{k})}^{(2)}=[X,T_{\tilde{f}_{\hat{I}_m^{(\hat{\nu},\hat{k})}}}]T_{\tilde{f}_{\hat{I}_m^{(\hat{\nu},\hat{k})}}} ~~\text{  and  }~~ X_{(\hat{\nu},\hat{k})}^{(2)}=T_{\tilde{f}_{\hat{I}_m^{(\hat{\nu},\hat{k})}}}XT_{\tilde{f}_{\hat{I}_m^{(\hat{\nu},\hat{k})}}}.$$
By
(\ref{m-1})
, we have \begin{equation}\label{Z2}\|Z_{(\hat{\nu},\hat{k})}^{(2)}\|\leq 2\varepsilon. \end{equation} Then note that $X_{(\hat{\nu},\hat{k})}^{(2)}=Y_{(\hat{\nu},\hat{k})}+Z_{(\hat{\nu},\hat{k})}^{(3)}$, where

\[
Y_{(\hat{\nu},\hat{k})}=\sum_{\hat{w}\in\hat{I}_m^{(\hat{\nu},\hat{k})}}T_{\hat{f}_{\hat{w}}}XT_{\tilde{f}_{\hat{w}}},~~Z_{(\hat{\nu},\hat{k})}^{(3)}=\sum_{\hat{w}'\neq \hat{w}\in\hat{I}_m^{(\hat{\nu},\hat{k})}}T_{\tilde{f}_{\hat{w}}}XT_{\hat{f}_{\hat{w}'}}.
\]
By Lemma \ref{Lemma 11}, we have $Y_{(\hat{\nu},\hat{k})}\in LOC(X)$. For any $k \in\{0, 1, 2, 3, 4, 5\}$, $\nu \in\{1,\cdots, \widetilde{N}_0\}$ and any $w=(m,6j+k,u)\in I_{m}^{(\nu,k)}$, we denote $|w|=6j+k$. To estimate the norm of $Z_{(\hat{\nu},\hat{k})}^{(3)}$, for $J\in\mathbb{N}$, we define

\[
\hat{I}_{m,J}^{(\hat{\nu},\hat{k})}=\{\hat{w}=(w_1,\cdots,w_n)\in \hat{I}_{m}^{(\hat{\nu},\hat{k})}:|w_i|\leq J \text{ for any } 1\le i \le n\}.
\]
Observe that for every $J\in \mathbb{N}$,
\[
\sum_{\hat{w}\neq \hat{w}'\in \hat{I}_{m,J}^{(\hat{\nu},\hat{k})}}M_{\tilde{f}_{\hat{w}}}XPM_{\tilde{f}_{\hat{w}'}}=\sum_{\hat{s}\in \hat{I}_{m,J}^{(\hat{\nu},\hat{k})}}M_{\chi_{\hat{B}_{\hat{s}}}}\sum_{\hat{w}\neq \hat{w}'\in \hat{I}_m^{(\hat{\nu},\hat{k})}}M_{\tilde{f}_{\hat{w}}}XPM_{\tilde{f}_{\hat{w}'}}\sum_{\hat{s}'\in \hat{I}_{m,J}^{(\hat{\nu},\hat{k})}}M_{\chi_{\hat{B}_{\hat{s}'}}}
\]
This leads to strong convergence in $B(L_a^2(\mathbb{D}))$,

$$\sum_{\hat{w}\neq \hat{w}'\in \hat{I}_{m,J}^{(\hat{\nu},\hat{k})}}M_{\tilde{f}_{\hat{w}}}XPM_{\tilde{f}_{\hat{w}'}}\xrightarrow{strongly} \sum_{\hat{w}\neq \hat{w}'\in \hat{I}_m^{(\hat{\nu},\hat{k})}}M_{\tilde{f}_{\hat{w}}}XPM_{\tilde{f}_{\hat{w}'}} \quad \text{as}\quad J\rightarrow \infty.$$

Compressing the strong convergence to the subset of $L_a^2(\mathbb{D})$, we see that there exists a $J\in\mathbb{N}$ such that
\begin{equation}\label{Z4}
\|Z_{(\hat{\nu},\hat{k})}^{(3)}\|\le 2\|Z_{(\hat{\nu},\hat{k})}^{(4)}\|,
\end{equation}
where
$$Z_{(\hat{\nu},\hat{k})}^{(4)}=\sum_{\hat{w}'\neq \hat{w}\in\hat{I}_{m,J}^{(\hat{\nu},\hat{k})}}T_{\tilde{f}_w}XT_{\tilde{f}_{w'}}.$$

Applying Lemma \ref{Lemma 13}, there exists
$\text{ a subset }\hat{I}_{0}\subseteq \hat{I}_{m,J}^{(\hat{\nu},\hat{k})}\text{ and coefficients }r_{\hat{w}}\in \mathbb{C}\text{ with }|r_{\hat{w}}|=1\text{ for any }\hat{w}\in \hat{I}_{m,J}^{(\hat{\nu},\hat{k})}$
 such that if we define
$$F=\sum_{\hat{w}\in \hat{I}_{0}}\tilde{f}_{\hat{w}}, \quad G=\sum_{\hat{w}\in \hat{I}_{m,J}^{(\hat{\nu},\hat{k})}\backslash \hat{I}_{0}}\tilde{f}_{\hat{w}},\quad F'=\sum_{\hat{w}\in \hat{I}_{0}}r_{\hat{w}}\tilde{f}_{\hat{w}},\quad G'=\sum_{\hat{w}\in \hat{I}_{m,J}^{(\hat{\nu},\hat{k})}\backslash \hat{I}_{0}}r_{\hat{w}}\tilde{f}_{\hat{w}},$$
then we have

$$\|Z_{(\hat{\nu},\hat{k})}^{(4)}\|\le 4(\|T_{F'}XT_G\|+\|T_{G'}XT_F\|).$$
By Lemma \ref{Lemma 10} and (\ref{bianchayiju3}), we have $F,G\in \Phi(\rho_m; n(\log m)^{-1}((6N_0)^n+2))$. Then $ \|[X,T_F]\|\leq 2\varepsilon.$ Note that $T_{G'}XT_F=T_{G'}[X,T_F]+T_{G'}T_FX$.
Combining with $T_{G'}T_F=-H_{\overline{G'}}^*H_F$, (\ref{bianchayiju3}) and Lemma \ref{Lemma 12}, we have

$$\|T_{G'}T_FX\|\le C_1\|X\|\text{diff}(F)\le C_1\|X\|\varepsilon \text{ for some } C_1>0.$$
Then we have
$$\|T_{G'}XT_F\|\le (2+C_1\|X\|)\varepsilon.$$
The similar argument also shows that
$$\|T_{F'}XT_G\|\le (2+C_2\|X\|)\varepsilon \text{ for some } C_2>0.$$
Combining this with (\ref{Z4}), we have

\begin{equation}\label{Z3}
\|Z_{(\hat{\nu},\hat{k})}^{(3)}\|\le 8(2+C_3\|X\|)\varepsilon \text{ for some }C_3>0.
\end{equation}

Looking back through the above discussion, for each $\hat{\nu}\in \{1,\dots,\widetilde{N}_0\}^n$ and each $\hat{k}\in \{0, 1,\dots,5\}^n$, we have the decomposition
$$X_{(\hat{\nu},\hat{k})}=Y_{(\hat{\nu},\hat{k})}+Z^{(1)}_{(\hat{\nu},\hat{k})}+Z^{(2)}_{(\hat{\nu},\hat{k})}+Z^{(3)}_{(\hat{\nu},\hat{k})},$$
where $Y_{(\hat{\nu},\hat{k})}\in LOC(X)$ and $Z^{(1)}_{(\hat{\nu},\hat{k})},\,Z^{(2)}_{(\hat{\nu},\hat{k})},\,Z^{(3)}_{(\hat{\nu},\hat{k})}$ satisfy estimates (\ref{Z1}), (\ref{Z2}) and (\ref{Z3}) respectlively. Combining this with (\ref{X}), we have that $X\in  \overline{Span\{LOC(X)+\mathcal{K}\}^{\|\cdot\|}}.$

\end{proof}
\section{The proof of Theorem C and D}
\

The purpose of this section is to give the proof of Theorem C and D.  To begin, we need the following lemmas.

\begin{lemma}\label{Lemma 4.1} For any $f\in VO_{bdd}$, then for every $R\ge 1$, we have \begin{equation}\label{vogongshi}\mathop{\mathrm{lim}}\limits_{t\rightarrow \infty}\mathrm{sup}\{|f(z)-f(w)|:\text{ for any }z,w\in \mathbb{D}\text{ with }d(z,w)\le R,d(0,z)\geq t>R\}=0.\end{equation}\end{lemma}

\begin{proof}
We claim that for any $k\in \mathbb{N}_+$, any $u,v\in D$,  if $\beta(u,v)\le k$, then there are $u_1,\dots, u_{k-1}\in D$ satisfying the condition $\beta(z,u_1)\le 1, \beta(u_1,u_2)\le1, \dots, \beta(u_{k-1},v)\le 1$.

If the claim was true, then for any $z=(z_1,\dots,z_n),\,w=(w_1,\dots,w_n)\in \mathbb{D}$ with $d(z,w)\le R$ and $d(0,z)\geq t>R$,
\begin{eqnarray*}
|f(z)-f(w)|&\le& |f(z)-f(z^{(1)})|+|f(z^{(1)})-f(z^{(2)})|+\dots+|f(z^{(n-1)})-f(w)|\\
&\le&([\beta(z_1,w_1)]+1+\dots+[\beta(z_n,w_n)]+1) \mathrm{diff}_{t-R}(f)\\
&\le& (R+n)\mathrm{diff}_{t-R}(f),
\end{eqnarray*}
where $z^{(1)}=(w_1,z_2,z_3,\dots,z_n),\, z^{(2)}=(w_1,w_2,z_3,\dots,z_n),\dots, z^{(n-1)}=(w_1,w_2,\dots,w_{n-1},z_n)$ and $d(z,z^{(i)})\le R$ for any $i$. Combined with $f\in VO_{bdd}$, it gives that the formula (\ref{vogongshi}) holds.

To complete the proof, it remains to check the claim. By the M\"{o}bius invariance of $\beta$, it suffices to show that for any $u,v\in D$ satisfying the condition $\varphi_u(v)\in \mathbb{R}_+$. If $\beta(0,\varphi_u(v))\leq k\in \mathbb{N}_+$, let $v_1=\frac{e^2-1}{e^2+1}$, then a straightforward computation shows that
$$\beta(0,v_1)=1\quad \text{and} \quad \beta(v_1,\varphi_u(v))=\frac12 \log \frac{(1-v_1)(1+\varphi_u(v))}{(1+v_1)(1-\varphi_u(v))}=\beta(0,\varphi_u(v))-\beta(0,v_1)\leq k-1.$$
Continue this process with $v_1=u$ and $\varphi_u(v)=v$, we would complete the proof.
\end{proof}

\begin{lemma}\label{Lemma 4.4}
If $g\in VO_{bdd}$, then $\lim_{z\to \partial \mathbb{D}}\|(g-g(z))k_z\|=0$.
\end{lemma}

\begin{proof}
For any $z\in\mathbb{D}$ and any $R>0$,
\[
\|(g-g(z))k_z\|^2=\int_{\widetilde{D}(z,R)}|g(w)-g(z)|^2|k_z(w)|^2dv(w)+\int_{\mathbb{D}\setminus \widetilde{D}(z,R)}|g(w)-g(z)|^2|k_z(w)|^2dv(w).
\]
By Lemma \ref{Lemma 4.1}, for any fixed $R>0$,
\[
\lim_{z\to \partial\mathbb{D}}\int_{\widetilde{D}(z,R)}|g(w)-g(z)|^2|k_z(w)|^2dv(w)=0.
\]
By the M\"{o}bius invariance of $\beta$ and $d\Lambda$, we have
\begin{eqnarray*}
\int_{\mathbb{D}\setminus \widetilde{D}(z,R)}|g(w)-g(z)|^2|k_z(w)|^2dv(w) &\leq & 2\|g\|_{\infty}\int_{\mathbb{D}\setminus \widetilde{D}(z,R)}|k_z(w)|^2dv(w) \\
&\leq & 2\|g\|_{\infty}\int_{\mathbb{D}\setminus \widetilde{D}(z,R)} (1-|\varphi_z(w)|^2)^2d\Lambda(w)\\
&\leq & 2\|g\|_{\infty}\int_{\mathbb{D}\setminus \widetilde{D}(0,R)} (1-|\varphi_z(\varphi_z(w))|^2)^2d\Lambda(\varphi_z(w))\\
&=&2\|g\|_{\infty}\int_{\mathbb{D}\setminus \widetilde{D}(0,R)} (1-|w|^2)^2d\Lambda(w).
\end{eqnarray*}
which tends to 0 as $R$ tends to $\infty$. This completes the proof.
\end{proof}

\begin{proposition}\label{Proposition 4.5}
Let $\{z^{(j)}\},\{w^{(j)}\}$ be sequences in $\mathbb{D}$ satisfying the following conditions:

(1) $z^{(j)},w^{(j)}\to \partial \mathbb{D}$ as $j\to\infty$

(2) There exists a $0<C<\infty$ such that $d(z^{(j)},w^{(j)})\leq C$ for every $j\in \mathbb{N}$. \\Then for every $A\in EssCom(\mathcal{T}(\mathbb{D}))$, we have
$$\lim_{j\to\infty}\|[A,k_{z^{(j)}}\otimes k_{w^{(j)}}]\|=0.$$
\end{proposition}

\begin{proof}

If not, there exists an operator $A\in EssCom(\mathcal{T}(\mathbb{D}))$ and two sequences $\{z^{(j)}\},\{w^{(j)}\}$ satisfying the conditions (1) and (2) such that
$$ \|[A,k_{z^{(j)}}\otimes k_{w^{(j)}}]\|\ge c, \text{ for some }c>0, \text{ for every } j\in \mathbb{N}. $$
For any $j\in\mathbb{N}$, it is easy to check that $k_{z^{(j)}}\otimes k_{w^{(j)}}\in \mathcal{D}_0$ with $\|k_{z^{(j)}}\otimes k_{w^{(j)}}\|\le 1$. Since both sequences $\{k_{z^{(j)}}\},\{k_{w^{(j)}}\}$ converge weakly to 0 in $L_a^2(\mathbb{D}),$
then the sequence $\{k_{z^{(j)}}\otimes k_{w^{(j)}}\}$ converges to 0 in the strong operator topology. For any $j\in\mathbb{N}$, let
$$B_j=[A,k_{z^{(j)}}\otimes k_{w^{(j)}}].$$
Then $B_j$ is compact with $\|B_j\|\le 2\|A\|$. And both sequences $\{B_j\}$ and $\{B_j^*\}$ converge to 0 in the strong operator topology. By taking subsequence if necessary, we can assume that $\{z^{(j)}\}$ is 1-separated and
$$\lim_{j\to\infty}\|B_j\|=d\ge c,\text{ for some }d>0.$$
Applying Lemma \ref{Lemma 6} to the sequence $\{B_j\}$, it gives us a subsequence $\{B_{j_v}\}$ such that the strong limit
$B=\lim_{N\to\infty}\sum_{v=1}^NB_{j_v}$ exists and $\|[A,\sum_{v=1}^{\infty}k_{z^{(j_v)}}\otimes k_{w^{(j_v)}}]\|_{\mathcal{Q}}=\|B\|_{\mathcal{Q}}=d\ge c>0$. It implies $B\notin \mathcal{K}$.

 Since $\{z^{(j_v)}\}$ is 1-separated and $d(z^{(j_v)},w^{(j_v)})\leq C$ for any $v\in\mathbb{N}$, then $Y=\sum_{v=1}^{\infty}k_{z^{(j_v)}}\otimes k_{w^{(j_v)}}\in \mathcal{D}_0$ and hence $Y\in \mathcal{T}(\mathbb{D})$ by Proposition \ref{Proposition J}, which contradicts with the statement $A\in EssCom(\mathcal{T}(\mathbb{D}))$.
\end{proof}

\begin{lemma}\label{Lemma 4.6}(\cite{J. Xia 2017},Lemma 5.1)
Let $T$ be a bounded, self-adjoint operator on a Hilbert space $\mathcal{H}$. Then for any unit vector $x\in \mathcal{H}$,
$$\|[T,x\otimes x]\|=\|(T-\langle Tx,x\rangle)x\|.$$
\end{lemma}

\begin{lemma}
\label{Lemma 4.7}(\cite{J. Xia 2017},Lemma 5.2)
Let $T$ be a bounded, self-adjoint operator on a Hilbert space $\mathcal{H}$. Then for any unit vector $x,y\in \mathcal{H}$,
$$|\langle Tx,x\rangle-\langle Ty,y\rangle|\leq \|[T,x\otimes x]\|+\|[T,x\otimes y]\|+\|[T,y\otimes y]\|.$$
\end{lemma}

\begin{proposition}\label{Proposition 4.8}
If $A\in EssCom(\mathcal{T}(\mathbb{D}))$, then its Berezin transform $\widetilde{A}\in VO_{bdd}$.
\end{proposition}

\begin{proof}
It suffices to consider self-adjoint bounded operator $A$. Obviously, $\widetilde{A}$ is bounded and continuous. If $\widetilde{A}\notin VO$, then there exist sequences $\{z^{(j)}\},\{w^{(j)}\}\subset \mathbb{D}$ and  $c>0$ such that:

(1) $d(z^{(j)},w^{(j)})\leq 1$ for any $j\in\mathbb{N}$,

(2) $|\langle Ak_{z^{(j)}},k_{z^{(j)}}\rangle-\langle Ak_{w^{(j)}},k_{w^{(j)}}\rangle|=|\widetilde{A}(z^{(j)})-\widetilde{A}(w^{(j)})|\geq c$ for any $j\in\mathbb{N}$,

(3) $z^{(j)}\to\partial \mathbb{D}$ as $j\to\infty$.\\
By Lemma \ref{Lemma 4.7}, we have
$$|\widetilde{A}(z^{(j)})-\widetilde{A}(w^{(j)})|\leq \|[A,k_{z^{(j)}}\otimes k_{w^{(j)}}]\|+\|[A,k_{z^{(j)}}\otimes k_{z^{(j)}}]\|+\|[A,k_{w^{(j)}}\otimes k_{w^{(j)}}]\|.$$
By Proposition \ref{Proposition 4.5}, the right side of the inequality approaches to 0 as $j\to\infty$, which contradicts with the condition (2).
\end{proof}

\begin{lemma}\label{Lemma 4.9}
If $A\in EssCom(\mathcal{T}(\mathbb{D}))$, then $\lim_{z\to \partial \mathbb{D}}\|(A-T_{\widetilde{A}})k_z\|=0$.
\end{lemma}

\begin{proof}
Without lose of generality, we assume $A=A^*$. By Lemma \ref{Lemma 4.6} and Proposition \ref{Proposition 4.5},
$$\lim_{z\to \partial \mathbb{D}}\|(A-\widetilde{A}(z))k_z\|=\lim_{z\to \partial \mathbb{D}}\|[A,k_z\otimes k_z]\|=0.$$
Therefore, it suffices to prove
$$\lim_{z\to \partial \mathbb{D}}\|(T_{\widetilde{A}}-\widetilde{A}(z))k_z\|=0.$$
Since by Proposition \ref{Proposition 4.8}, $\widetilde{A}\in VO_{bdd}$, then  by Lemma \ref{Lemma 4.4}, we have  $\lim_{z\to \partial \mathbb{D}}\|(\widetilde{A}-\widetilde{A}(z))k_z\|=0$. Combined with
$$\|(T_{\widetilde{A}}-\widetilde{A}(z))k_z\|=\|P(\widetilde{A}-\widetilde{A}(z))k_z\|\leq \|(\widetilde{A}-\widetilde{A}(z))k_z\|,$$
it gives that $\lim_{z\to \partial \mathbb{D}}\|(T_{\widetilde{A}}-\widetilde{A}(z))k_z\|=0$.
\end{proof}

\begin{lemma}\label{Lemma 4.3} For the orthogonal projection $P:L^2(\mathbb{D})\rightarrow L^2_a(\mathbb{D})$ and any $f\in VO_{bdd}$, $[M_f,P]\in \mathcal{K}$.

\end{lemma}

\begin{proof}
For each $f\in VO_{bdd}$ and any $\varepsilon>0$, there exists a $0<R<1$ such that
\[
\sup\{|f(z)-f(w)|:\forall w\in\mathbb{D}\text{ with }d(z,w)\le 1\}<\epsilon \text{ for any }z \in\mathbb{D}\backslash \mathbb{D}_R,
\]
where $\mathbb{D}_R$ appears in Definition \ref{banjing}.
For $v(z)\in \Psi(R;\epsilon)$, we define $g_R(z)=f(z)v(z)$ and define $h_R(z)=f(z)-g_R(z)$.

Then $g_R$ has compact support and diff$(h_R)<2\epsilon$, which follows that $[M_{g_R},P]\in \mathcal{K}$ and we have $\|[M_{h_R},P]\|<C\epsilon$ for some $C>0$ by Lemma \ref{Lemma 12}.
\end{proof}
\begin{remark}\label{Remark 1}
Combining Proposition \ref{Proposition J}, Proposition \ref{Proposition K}, Proposition \ref{Proposition L} and Proposition \ref{Proposition M}, we have
\[
\text{EssCom}\{T_g:g\in VO_{bdd}\}\subset \mathcal{T}(\mathbb{D}).
\]
\end{remark}

\noindent
{\bf Proof of Theorem C:} To prove $EssCom\{T_g:g\in VO_{bdd}\}\supset \mathcal{T}(\mathbb{D})$, it suffices to prove that for any $f\in L^{\infty}(\mathbb{D})$,
\[
T_fT_g-T_gT_f\in \mathcal{K}\text{ for any }g\in VO_{bdd}.
\]
By Lemma \ref{Lemma 4.3}, for the projection $P$ onto $L_a^2(\mathbb{D})$, $[P,M_g]\in \mathcal{K}$.

Therefore, for any $g\in VO_{bdd}$ and any $f\in L^{\infty}(\mathbb{D})$,
\begin{equation}\label{N1}
T_fT_g-T_gT_f=[P,M_f](P-1)[M_g,P]-[P,M_g](P-1)[M_f,P]\in \mathcal{K},
\end{equation}
which induces $EssCom\{T_g:g\in VO_{bdd}\}\supset \mathcal{T}(\mathbb{D})$.

On the other hand, since $\mathcal{K}\subset EssCom\{T_g:g\in VO_{bdd}\}$, then by Remark \ref{Remark 1},
$$ EssCom\{T_g:g\in VO_{bdd}\}= \mathcal{T}(\mathbb{D}).\Box$$

\begin{remark}\label{Remark 2}
Since $\{T_g:g\in VO_{bdd}\}\subset \mathcal{T}(\mathbb{D})$, then EssCom$\{T_g:g\in VO_{bdd}\}\supset$ EssCom$(\mathcal{T}(\mathbb{D}))$, which implies $\mathcal{T}(\mathbb{D})\supset$ EssCom$(\mathcal{T}(\mathbb{D}))$ by Theorem C. But obviously, we have $\mathcal{T}(\mathbb{D})\subset$ EssCom$(\mathcal{T}(\mathbb{D}))$. Therefore,
\[
\text{EssCom}(\mathcal{T}(\mathbb{D}))=\mathcal{T}(\mathbb{D}).
\]
\end{remark}

\noindent
{\bf Proof of Theorem D:} To prove EssCom$(\mathcal{T}(\mathbb{D}))\supset\{T_g:g\in VO_{bdd}\}+\mathcal{K}$, it suffices to prove
\[
\text{EssCom}(\mathcal{T}(\mathbb{D}))\supset\{T_g:g\in VO_{bdd}\}
\]
i.e.
\[T_fT_g-T_gT_f\in \mathcal{K},\, \text{for every}\, f\in VO_{bdd}, g\in L^{\infty}(\mathbb{D}),
\]
which has already been proven in (\ref{N1}).

For the opposite direction. For $A\in$ EssCom$(\mathcal{T}(\mathbb{D}))$, by Lemma \ref{Lemma 4.9}, we have
$$\lim_{z\to \partial \mathbb{D}}\langle(A-T_{\widetilde{A}})k_z,k_z\rangle=0.$$
Combining Remark \ref{Remark 2} with Theorem A and Theorem B, we have $A-T_{\widetilde{A}}\in\mathcal{K}$. By Proposition \ref{Proposition 4.8}, we have $\widetilde{A}\in VO_{bdd}$, which completes the proof.
$\Box$

\end{document}